\documentclass[12pt,a4paper]{amsart}
\usepackage[top=35mm, bottom=35mm, left=30mm, right=30mm]{geometry}
\usepackage{graphicx}
\usepackage{amsthm,amscd,amsmath,amsthm,amssymb,amsfonts}
\usepackage[all,cmtip]{xy}
\usepackage{color}
\usepackage[colorlinks=true,citecolor=blue]{hyperref}
\usepackage{appendix}
\usepackage{lineno}
\newtheorem{theorem}{Theorem}[section]
\newtheorem{lemma}[theorem]{Lemma}
\theoremstyle{definition}
\newtheorem{definition}[theorem]{Definition}
\newtheorem{cor}[theorem]{Corollary}
\newtheorem{remark}[theorem]{Remark}
\newtheorem{prop}[theorem]{Proposition}

\def \RP {{\bf RP}}
\def \Z  {\mathbb{Z}}
\def \N  {\mathbb{N}}
\def \Q {{\bf Q}}

\begin{document}

\title{Top-nilpotent enveloping semigroups and pro-nilsystems}

\author[J.~Qiu]{Jiahao Qiu}
\address[J.~Qiu]{Wu Wen-Tsun Key Laboratory of Mathematics, USTC, Chinese Academy of Sciences and
School of Mathematics, University of Science and Technology of China,
Hefei, Anhui, 230026, P.R. China}
\email{qiujh@mail.ustc.edu.cn}

\author[J.~Zhao]{Jianjie Zhao}
\address[J.~Zhao]{Wu Wen-Tsun Key Laboratory of Mathematics, USTC, Chinese Academy of Sciences and
School of Mathematics, University of Science and Technology of China,
Hefei, Anhui, 230026, P.R. China}
\email{zjianjie@mail.ustc.edu.cn}

\begin{abstract}
In this paper, it is shown that for $d\in \mathbb{N}$, a minimal system $(X,T)$ is a $d$-step pro-nilsystem if
its enveloping semigroup is a $d$-step top-nilpotent group, answering an open question by Donoso. Thus,
combining the previous result of Donoso, it turns out that a minimal system $(X,T)$ is a $d$-step pro-nilsystem if
and only if its enveloping semigroup is a $d$-step top-nilpotent group.
\end{abstract}
\date{\today}

\maketitle

\section{Introduction}
By a \emph{topological dynamical system} or just a \emph{dynamical system},
we mean a pair $(X,G)$, where $X$ is a compact metric space with a metric $\rho$
and $G$ acts on it as a group of homeomorphisms.
In this paper, we only focus on abelian group actions.
When $G$ is the group induced by some homeomorphism $T$,
we just write it as $(X,T)$.

\subsection{The history of the question}
In order to study the asymptotic behaviors of a dynamical system $(X,G)$,
Ellis introduced in 1960 the enveloping semigroup $E(X,G)$
which has been proved to be a very powerful tool in the theory of topological dynamical systems.
It is defined as the closure of the set $\{g:g\in G\}$ in $X^X$ (with its compact, usually non-metrizable, pointwise
convergence topology). Ellis showed that a dynamical system $(X,G)$ is equicontinuous
if and only if $E(X,G)$ is a group of homeomorphisms,
and $(X,G)$ is distal if and only if $E(X,G)$ is a group.
Furthermore, when $(X,G)$ is minimal,
then it is equicontinuous if and only if $E(X,G)$ is an abelian group.
So, it is natural to ask: can we give a finer classification of minimal distal systems
using enveloping semigroups?
This is the main motivation of the current paper.

\medskip

In the recent years, the study of the dynamics of rotations on nilmanifolds
and inverse limits of this kind of dynamics has drawn much interest, since it relates
to many dynamical properties and has important applications in number theory.
We refer to \cite{HK06} and the references therein for a systematic treatment of the subject.

\medskip

In a pioneer work, Host, Kra and Maass in \cite{HKM} introduced the notion of
{\it regionally proximal relation of order $d$}
for a dynamical system $(X,T)$, denoted by $\mathbf{RP}^{[d]}(X)$.
For $d\in\N$, we say that a minimal system is a \emph{system of order d}
if $\mathbf{RP}^{[d]}(X)=\Delta$ and this is equivalent for $(X,T)$ being a $d$-step pro-nilsystem,
i.e. an inverse limit of rotations on $d$-step nilsystems (see \cite[Theorem 2.8]{HKM}).
For a minimal distal system $(X,T)$, it was proved that
$\mathbf{RP}^{[d]}(X)$ is an equivalence relation and $X/\mathbf{RP}^{[d]}(X)$
is the maximal factor of order $d$ \cite{HKM}.
Then Shao and Ye \cite{SY-12} showed that in fact for
any minimal system, $\mathbf{RP}^{[d]}(X)$ is an equivalence
relation and $\mathbf{RP}^{[d]}(X)$
has the so-called lifting property.

We note that for general group actions,
regional proximality of higher order can also be defined.
Meanwhile,
using similar methods,
the result obtained by Shao and Ye \cite{SY-12}
can be generalized to general group actions \cite{GGY}.

\medskip

An earlier open question is the following: is it true that a minimal system
$(X,T)$ is a $d$-step pro-nilsystem if and only if its enveloping
semigroup is a $d$-step nilpotent group? It is Glasner who first considered the
question. In \cite{EG3},  Glasner proved that the question has an affirmative answer,
when $d=2$ and the system is an extension of its maximal equicontinuous factor by a torus.
In \cite{SD14}, Donoso shown that the enveloping semigroup of a $d$-step pro-nilsystem
is a $d$-step top-nilpotent group
and hence a $d$-step nilpotent group.
Note that an Ellis group $E$ is \emph{top-nilpotent},
 if the descending sequence
 $E=E_1^{\mathrm{top}} \supset E_2^{\mathrm{top}} \supset \ldots$ terminates at some point,
 where $E_{j+1}^{\mathrm{top}}$ is defined as the closure of the subgroup spanned by $[E_j^{\text{top}},E],j\geq1$.
 \footnote{Actually in the current paper, we consider the smallest closed subgroup generated by the commutators, but we
prove later that this coincides with the closure of the group.}
By using the notation of top-nilpotency, Donoso \cite{SD14} proved that a minimal system with a 2-step
top-nilpotent enveloping semigroup has to be a 2-step pro-nilsystem.

\medskip

So, it seems that a more suitable question according to the result of Donoso is
that (Question 1.3 in \cite{SD14}):
Let $(X, T)$ be a minimal system with a $d$-step top-nilpotent enveloping
semigroup with $d >2$.
Is $(X, T)$ a $d$-step pro-nilsystem?
We will address this question in the current paper and give an affirmative answer.

\subsection{The main result}
Now we state the main result of the paper and describe briefly how we obtain it.
Host, Kra and Maass \cite{HKM} provided  a method to describe
regional proximality of higher order by using dynamical cubespaces
which are also called dynamical parallelepipeds. Assume $(X, T)$ is a minimal system with a $d$-step
top-nilpotent enveloping semigroup. To show that $(X,T)$ is a $d$-step pro-nilsystem,
the first step we do is to show that $(X,T)$ is a system of order $\infty$. Then in the second
step we prove that indeed $(X,T)$ is a $d$-step pro-nilsystem.

\medskip

In the proof of the first step,
we study general top-nilpotent Ellis groups and their Host-Kra cubegroups.
Let $E$ be a $d$-step top-nilpotent Ellis group,
then $E^{[l]}=E^{2^l}$ is also a $d$-step top-nilpotent Ellis group with the product topology for every $l\in \N$.
In \cite[Chapter 12]{HK06}, Host-Kra proved that the Host-Kra cubegroups of nilpotent Lie groups are also nilpotent
Lie groups.
Unfortunately, it does not hold for top-nilpotent Ellis groups.
That is, the Host-Kra cubegroup $\mathcal{HK}^{[l]}(E)$ is not an Ellis group.
We denote by $\widetilde{E}^{[l]}$
the closed subgroup generated by $\mathcal{HK}^{[l]}(E)$,
which is also a $d$-step top-nilpotent Ellis group.
To study the Ellis group $\widetilde{E}^{[l]}$ and its topological commutators subgroups, the binary cubegroups $\{\mathcal{C}_j^{[l]}(E)\}_{j=1}^d$ are introduced
 which are much easier to handle.
It is shown that  for every $j=1,\ldots,d,( \widetilde{E}^{[l]})_j^{\mathrm{top}}$
is included in the group $\mathcal{C}_j^{[l]}(E)$.
Using this result,
we show that the enveloping semigroups of the dynamical cubespaces
associated with $(X,T)$ are also top-nilpotent.
Moreover, the corresponding topological commutators subgroups
have precise forms.
So there is some restriction on the order of regional proximality,
meaning that the system cannot admit nontrivial regionally proximal pairs of order $\infty$.
The proof of the first step indicates that
we can restrict to the case that $\mathbf{RP}^{[\infty]}(X)=\Delta$.
Note that a system of order $\infty$ is an inverse limit of
minimal nilsystems (see \cite[Theorem 3.6]{PD13}).
 Since the inverse limit is easy to handle, we need only to
focus on minimal nilsystems.

In the proof of the second step, we study the Furstenberg tower of a minimal nilsystem and
show that
this tower coincides with the maximal factors of higher order.
Following these facts, we deduce that the question has an affirmative answer.

\medskip

Now we state the main result and the results we need to get it.
As a parallelepipeds group
is generated by several commuting transformations,
we consider
general abelian groups action.

\begin{theorem}\label{main-thm1}
For abelian group action,
a minimal system with a top-nilpotent enveloping semigroup is a system of order $\infty$.
\end{theorem}

The remarkable theorem of Furstenberg on minimal distal systems states that a minimal distal system
is the inverse limit of isometric extensions, which we will refer as the Furstenberg tower. The following result shows
that for a minimal nilsystem, the Furstenberg tower and the maximal factors of
higher order coincide.
This fact seems easy to prove, but in fact the proof is much involved.

\begin{theorem}\label{main-thm2}
  Let $s\geq 3$ be an integer and let $(X,T)$ be a minimal $s$-step nilsystem.
Then the extension $X\to X/\mathbf{RP}^{[s-1]}(X)$ is isometric
and the maximal isometric extension of $X/\mathbf{RP}^{[r]}(X)$ below $X$
is $X/\mathbf{RP}^{[r+1]}(X),r=1,\ldots,s-2$.
\end{theorem}

Given a minimal system $(X,T)$ with a $d$-step top-nilpotent enveloping semigroup,
let
$\mathbf{R}_j(X)=\{(x,px):x\in X,p\in E_{j+1}^{\mathrm{top}}(X)\}, j=1,\ldots,d.$
We can show that indeed $\mathbf{R}_j(X)$ is a closed invariant equivalence relation
and $X/\mathbf{R}_j(X)$  is the maximal factor of $X$
which has a $j$-step top-nilpotent enveloping semigroup.
Moreover it turns out that the extension
$X/\mathbf{R}_{j+1}(X)\to X/\mathbf{R}_{j}(X)$ is isometric.

Following Theorem \ref{main-thm2}, for a minimal nilsystem,
the Furstenberg tower, the maximal factors of higher order and the factors defined
above all coincide.
By Theorem \ref{main-thm1}, we show that such property holds for any minimal
system with a top-nilpotent enveloping semigroup,
 answering the open question asked by Donoso.

\begin{theorem}\label{main-thm3}
Let $(X,T)$ be a minimal system and $d\in \N$.
Then, it is a $d$-step pro-nilsystem if it has a $d$-step top-nilpotent enveloping semigroup.
\end{theorem}

Note that the converse statement was proved in \cite{SD14}. Thus, we conclude that a minimal system
is a $d$-step pro-nilsystem if and only if it has a $d$-step top-nilpotent enveloping semigroup.

As a consequence, we obtain an interesting result immediately.

\begin{theorem}\label{main-thm4}
Let $(X,T)$ be a minimal system with a $d$-step top-nilpotent enveloping semigroup, where $d\in \N$.
For $k=1,\ldots,d$ and points $x,y\in X$,
we have that $(x,y)\in \mathbf{RP}^{[k]}(X)$
if and only if there is some element $p\in E_{k+1}^{\mathrm{top}}(X)$
with $y=px$.
\end{theorem}

\subsection{Organization of paper}
The paper is organized as follows.
In Section 2, the basic notions used in the paper are introduced.
In Section 3,
a description of the maximal factors of order $\infty$ of dynamical cubespaces is given.
In Section 4,
we study general top-nilpotent Ellis groups and their Host-Kra cubegroups.
In Section 5,
we show that
for abelian group action,
any minimal system with a top-nilpotent enveloping semigroup is a
system of order $\infty$ (Theorem \ref{main-thm1}).
In Section 6,
we prove that for a minimal nilsystem, the Furstenberg tower and the maximal factors of higher order
coincide (Theorem \ref{main-thm2}).
In the final section,
we give proofs of Theorems \ref{main-thm3} and \ref{main-thm4}.

\medskip

\noindent {\bf Acknowledgments.}
The authors would like to thank Professors W. Huang, S. Shao,  X. Ye and Dr. F. Cai
for helping discussions and remarks. We thank Professor V. Bergelson for bringing our
attention to this problem again.
The authors were supported by NNSF of China (11431012).

\section{Preliminaries}
In this section we gather definitions and preliminary results that
will be necessary later on.
Let $\N$ and $\Z$ be the sets of all positive integers
and integers respectively.
Let $F$ be a finite set and denote by $|F|$ the number of the elements of $F$.

\subsection{Topological dynamical systems}
Throughout the paper, $(X,G)$ denotes a \emph{topological dynamical system}
(or \emph{dynamical system})
where $X$ is a compact metric space with a metric $\rho(\cdot,\cdot)$
and $G$ is a countable abelian group.
For $x\in X,\mathcal{O}(x,G)=\{gx:g\in G\}$ denotes the \emph{orbit} of $x$.
A dynamical system $(X,G)$ is called \emph{minimal} if
every point has dense orbit in $X$.
When $G$ is induced by some homeomorphism $T$,
i.e. $G=\{T^n:n\in\Z\}$,
we write it as $(X,T)$.

A \emph{homomorphism} $\pi :X\to Y$ between systems $(X,G)$ and $(Y,G)$
is a continuous onto map such that $\pi \circ g=g \circ \pi$ for every $g\in G$;
one says that $(Y,G)$ is a \emph{factor} of $(X,G)$
and that $(X,G)$ is an \emph{extension} of $(Y,G)$, and one also refers to $\pi$
as a \emph{factor map} or an \emph{extension}. The systems are said to be \emph{conjugate} if $\pi$ is bijective.
An extension $\pi $ is determined by the corresponding closed invariant equivalence relation
$R_\pi=\{   (x_1,x_2):\pi (x_1)=\pi(x_2)\}$.

\medskip

Suppose that we have a compact group $K$ of homeomorphisms of $X$ commuting with $G$, i.e. $kg=gk$ for
any $g\in G$ and $k\in K$ (where $K$ is endowed with the topology of uniform convergence).
Let $R_K=\{(x,kx):x\in X,k\in K\}$  and it is an equivalence relation.
We can define a factor $Y$ by setting $Y=X/R_K$ and we say that $(X,G)$
is an extension of $(Y,G)$ by the group $K$.

Let $\pi:X\to Y$ be a factor map between minimal systems $(X,G)$ and $(Y,G)$.
We say that $(X,G)$ is an isometric extension of $(Y,G)$ if for any $\varepsilon>0$,
there exists $\delta>0$ such that $\rho(gx,gy)<\varepsilon$ for all $g\in G$
whenever $\pi(x)=\pi(y)$ and $\rho(x,y)<\delta$.

Note that any group extension is isometric.

\subsection{Discrete cubes and faces}
For an integer $l\ge0$, we denote the set of maps $\{0,1\}^{l}\to X$
by $X^{[l]}(X^{[0]}=X)$.
For $x\in X$, write $x^{[l]}=(x,x,\ldots,x)\in X^{[l]}$.
The \emph{diagonal} of $X^{[l]}$ is $\Delta^{[l]}=\Delta^{[l]}(X)=\{ x^{[l]}:x\in X\}$.
Usually, when $l=1$, denote the diagonal by $\Delta_X$ or $\Delta$
instead of $\Delta^{[1]}$.
We can isolate the first coordinate,
writing $X^{[l]}_*=X^{2^l-1}$ and writing $\mathbf{x}\in X^{[l]}$
as $\mathbf{x}=(\mathbf{x}({\vec{0}}),\mathbf{x}_*)$,
where $\mathbf{x}_*=(\mathbf{x}(\epsilon):\epsilon\in \{0,1\}^l\backslash \{ \vec{0}\})\in X^{[l]}_*$.

Identifying $\{0,1\}^l$ with the set of vertices of the Euclidean unit cube,
an Euclidean isometry of the unit cube permutes the vertices of the
cube and thus the coordinates
of a point $\mathbf{x}\in X^{[l]}$.
These permutations are the \emph{Euclidean permutations} of $X^{[l]}$.

\medskip

A set of the form
\[
F=\{\epsilon\in\{0,1\}^{l}: \epsilon_{i_1}=a_{1},\ldots,\epsilon_{i_{k}}=a_{k}\}
\]
for some $k\geq0$, $1\leq i_{1}<\ldots<i_{k}\leq l$ and
 $a_{i}\in\{0,1\}$ is called a \emph{face} of \emph{codimension} $k$ of the discrete
cube $\{0,1\}^{l}$.\footnote{The case $k=0$ corresponds to $\{0,1\}^{d}$.}
One writes $\mathrm{codim}(F)=k$.
A face of codimension 1 is called a \emph{hyperface}.
If all $a_i=1$ we say that the face is \emph{upper}.
Note that all upper faces contain $\vec{1}$ and there are exactly $2^l$ upper faces.

\subsection{Host-Kra cubegroups}
Let $H$ be a group and let $F$ be a face of $\{0,1\}^l$.
For $h\in H$ we denote by $h^{(F)}$ the element of $H^{[l]}$
defined as $h^{(F)}(\epsilon)=h$ if $\epsilon\in F$ and $h^{(F)}(\epsilon)=e$ otherwise,
where $e$ denotes the unit element of $H$.
We call the subgroup
of $H^{[l]}$ generated by $h^{(F)}$,
where $h\in H$ and $F$ ranges over all hyperfaces of $\{0,1\}^l$,
the \emph{Host-Kra cubegroup} and denote it by $\mathcal{HK}^{[l]}(H)$.
We call the subgroup
of $H^{[l]}$ generated by $h^{(F)}$,
where $h\in H$ and $F$ ranges over all upper hyperfaces of $\{0,1\}^l$,
the \emph{face cubegroup} and denote it by $\mathcal{F}^{[l]}(H)$.
It is easy to see that $\mathcal{HK}^{[l]}(H)$ is generated by $\mathcal{F}^{[l]}(H)$ and $\Delta^{[l]}(H)$.

\medskip

The Host-Kra and face cubegroups originate in \cite[Section 5]{HK05}
and coincide with the parallelepiped groups and face groups respectively of \cite[Definition 3.1]{HKM}
introduced for abelian actions.
 As a matter of fact, if $H$ is an abelian group,
 then for any $\mathbf{g}\in \mathcal{HK}^{[l]}(H)$ there exist $h\in H$ and $\vec{h}=(h_1,\ldots,h_l)\in H^l$
 such that
 \begin{equation}\label{rep-hk}
   \mathbf{g}=(h+\vec{h}\cdot \epsilon:\epsilon\in \{0,1\}^l),
 \end{equation}
where $\vec{h}\cdot \epsilon=\sum_{i=1}^{l}h_i\cdot \epsilon_i$.
Moreover, if $\mathbf{g}\in \mathcal{F}^{[l]}(H)$,
then $h=0$ in (\ref{rep-hk}).

\subsection{Dynamical cubespaces}
Let $(X,G)$ be a dynamical system and $l\in \N$.
We often write $\mathcal{F}^{[l]}$ and $\mathcal{G}^{[l]}$
instead of $\mathcal{F}^{[l]}(G)$ and $\mathcal{HK}^{[l]}(G)$.
The Host-Kra cubegroups act (coordinatewise) on
$X^{[l]}$ by
\[
(\mathbf{g}\mathbf{x})(\epsilon)=\mathbf{g}(\epsilon)\mathbf{x}(\epsilon)
\]
for $\mathbf{g}\in \mathcal{G}^{[l]},\mathbf{x}\in X^{[l]}$ and $\epsilon\in\{0,1\}^l$.

Let $\Q^{[l]}(X)=\overline{\{ \mathbf{g} x^{[l]}:x\in X,\mathbf{g} \in \mathcal{F}^{[l]}\}}$.
We call this set a \emph{dynamical cubespace of dimension $l$} of the
dynamical system $(X,G)$.
It is important to note that $\mathbf{Q}^{[l]}(X)$ is invariant under the Euclidean permutations of $X^{[l]}$.

For convenience, we denote the orbit closure of $\mathbf{x}\in X^{[l]}$
under $\mathcal{F}^{[l]}$ by $\overline{\mathcal{F}^{[l]}}(\mathbf{x})$,
instead of $\overline{\mathcal{O}(\mathbf{x},\mathcal{F}^{[l]})}$.
Let $\Q^{[l]}_x(X)=\Q^{[l]}(X)\cap (\{x\}\times X^{2^l-1})$.

\begin{theorem}\cite{SY-12}\label{property}
  Let $(X,G)$ be a minimal system and $l\in \N$. Then
  \begin{enumerate}
  \item $(\mathbf{Q}^{[l]}(X),\mathcal{G}^{[l]})$ is a minimal system.
    \item $(\overline{\mathcal{F}^{[l]}}(x^{[l]}),\mathcal{F}^{[l]})$ is minimal for all $x\in X$.
    \item $\overline{\mathcal{F}^{[l]}}(x^{[l]})$ is the unique $\mathcal{F}^{[l]}$-minimal
    subset in $\mathbf{Q}^{[l]}_x(X)$ for all $x\in X$.
  \end{enumerate}
\end{theorem}

\subsection{Proximality and regional proximality of higher order}
Let $(X,G)$ be a dynamical system. A pair $(x,y)\in X\times X$ is a \emph{proximal} pair if
\[
\inf_{g\in G}\rho(gx,gy)=0
\]
and a distal pair if it is not proximal.
Denote by $\mathbf{P}(X)$ the set of all proximal pairs of $X$.
The dynamical system $(X,G)$ is \emph{distal} if $(x,y)$ is a distal pair whenever $x,y\in X$ are distinct.

\begin{definition}\label{def-rp}
Let $(X,G)$ be a dynamical system and $d\in \N$.
   The \emph{regionally proximal relation of order $d$} is the relation $\textbf{RP}^{[d]} (X)$
defined by: $(x,y)\in\textbf{RP}^{[d]}(X)$ if
and only if for every $\delta>0$, there
exist $x',y'\in X$ and $\vec{g}\in G^d$ such that:
$\rho(x,x')<\delta,\rho(y,y')<\delta$ and
\[
\rho(  (\vec{g}\cdot\epsilon) x', (\vec{g}\cdot\epsilon)  y')<\delta
\]
for all $\epsilon\in \{0,1\}^d\backslash\{ \vec{0}\}$.

We say $(X,G)$ is a \emph{system of order $d$}
if its regionally proximal relation of order $d$ is trivial.
\end{definition}

\begin{remark}
Note that Definition \ref{def-rp} is suitable for all dynamical systems.
In the current paper, we only focus on minimal systems.
Thus, when we say a dynamical system is a system of order $d$,
always assume that it is minimal.
\end{remark}

It is easy to see that $\mathbf{RP}^{[d]}(X)$ is a closed and invariant relation.
It follows from \cite[Lemma 3.5]{SY-12} that
\[\mathbf{P}(X)\subset
\ldots\subset \mathbf{RP}^{[d+1]}(X)\subset \mathbf{RP}^{[d]}(X)\subset
\ldots \subset\mathbf{RP}^{[2]}(X)\subset \mathbf{RP}^{[1]}(X).
\]

\begin{theorem}\cite{SY-12}\label{cube-minimal}
  Let $(X,G)$ be a minimal system and $d\in \N$.
  Then,
  \begin{enumerate}
    \item $(x,y)\in \mathbf{RP}^{[d]}(X)$ if and only if $(x,y,y,\ldots,y)=(x,y^{[d+1]}_*)\in \mathbf{Q}^{[d+1]}(X)$
    if and only if $(x,y,y,\ldots,y)=(x,y^{[d+1]}_*)\in \overline{\mathcal{F}^{[d+1]}}(x^{[d+1]})$.
    \item $\mathbf{RP}^{[d]}(X)$ is an equivalence relation.
  \end{enumerate}
\end{theorem}

The regionally proximal relation of order $d$ allows us to construct the maximal
factor of order $d$
of a minimal system. That is, any factor of order $d$
factorizes through this system.

\begin{theorem}\label{lift-property}\cite{SY-12}
  Let $\pi :(X,G)\to (Y,G)$ be the factor map between minimal systems and $d\in \N$. Then,
  \begin{enumerate}
    \item $(\pi \times \pi )\mathbf{RP}^{[d]}(X)=\mathbf{RP}^{[d]}(Y)$.
    \item $(Y,G)$ is a system of order $d$ if and only if $\mathbf{RP}^{[d]}(X)\subset R_\pi$.
  \end{enumerate}

In particular, the quotient of $(X,G)$ under $\mathbf{RP}^{[d]}(X)$
is the maximal factor of order $d$ of $X$.
\end{theorem}

It follows that for any minimal system $(X,G)$,
\[
\mathbf{RP}^{[\infty]}(X)=\bigcap_{d\geq1}\mathbf{RP}^{[d]}(X)
\]
is a closed invariant equivalence relation.

Now we formulate the definition of systems of order $\infty$.

\begin{definition}
    A minimal system $(X,G)$ is a \emph{system of order $\infty$},
  if the equivalence relation $\mathbf{RP}^{[\infty]}(X)$ is trivial,
  i.e. coincides with the diagonal.
\end{definition}

\begin{remark}\label{infi}
  By using similar methods,
  we can show that Theorem \ref{lift-property} also holds for $d=\infty$.
\end{remark}

The following result can be found in \cite[Corollary 4.2]{HKM} for a single transformation.
For completeness, we include the proof.
\begin{prop}\label{fiber-property}
   Let $(X,G)$ be a minimal distal system and let $d\geq2$ be an  integer.
  Let $\pi:X\to X/\mathbf{RP}^{[d-1]}(X)$ be the factor map.
  If points $x_1,\ldots,x_{2^d}\in X$ satisfy $$\pi(x_1)=\ldots=\pi(x_{2^d}),$$
  then $( x_1,\ldots,x_{2^d}) \in \mathbf{Q}^{[d]}(X)$.
\end{prop}

\begin{proof}
For $d\in \N$ and $x\in X$, let
\[
\RP^{[d]}[x]=\{y\in X:(x,y)\in \mathbf{RP}^{[d]}(X)\}.
\]

We show it by induction on $d$.

When $d=2$.
Let $x_i\in X$ with $x_i\in \RP^{[1]}[x],i=1,\ldots,4$.
As the system $(X,G)$ is minimal,
there exists some sequence $\{g_i\}_{i\in \N}\subset G$ such that $g_ix_2\to x_1,i\to \infty$.
Without loss of generality, assume that $g_i x_4\to x_4',i\to \infty$
for some $x_4'\in X$. Thus
\[
(\mathrm{id}\times g_i\times  \mathrm{id}  \times g_i)(x_1,x_2,x_3,x_4)\to (x_1,x_1,x_3,x_4'),i\to \infty.
\]
By equivalence of $\mathbf{RP}^{[1]}(X)$, we get that $x_4'\in \RP^{[1]}[x]$.

Similarly, there exists some sequence $\{h_i\}_{i\in \N}\subset G$ such that $h_ix_3\to x_1,i\to \infty$.
Without loss of generality,
assume that $h_ix_4'\to x_4'',i\to \infty$ for some $x_4''\in X$.
Thus
\[
(\mathrm{id}\times \mathrm{id}\times h_i  \times h_i)(x_1,x_1,x_3,x_4')\to (x_1,x_1,x_1,x_4''),i\to \infty.
\]
By equivalence of $\mathbf{RP}^{[1]}(X)$, we get that $x_4''\in \RP^{[1]}[x]$.
This shows that
\[
\mathbf{x}''\in\overline{\mathcal{O}(\mathbf{x}',\mathcal{G}^{[2]})}\subset \overline{\mathcal{O}(\mathbf{x}_0,\mathcal{G}^{[2]})},
\]
where $\mathbf{x}_0=(x_1,x_2,x_3,x_4),\mathbf{x}'=(x_1,x_1,x_3,x_4')$
and $\mathbf{x}''=(x_1,x_1,x_1,x_4'')$.
By Theorem \ref{cube-minimal}, $\mathbf{x}''\in \Q^{[2]}(X)$.
Note that the system $(X,G)$ is distal, so is $(X^{[2]},\mathcal{G}^{[2]})$.
Thus we obtain that $\mathbf{x}_0\in \overline{\mathcal{O}(\mathbf{x}'',\mathcal{G}^{[2]})},$
which implies $\mathbf{x}_0\in \Q^{[2]}(X)$.

\medskip

Let $d>2$ be an integer and suppose the statement is true for all $j=2,\ldots,d-1$.
Let $x_i\in X$ with $x_i\in \RP^{[d-1]}[x],i=1,\ldots,2^d$ and
put $\mathbf{x}=(x_1,\ldots,x_{2^d})$.

For $i=1,\ldots,d$ and $s=0,1$,
let
\[
F_i^s=\{\epsilon\in \{0,1\}^d:\epsilon_i=s\}.
\]
Inductively we will construct elements
$\mathbf{c}_{1},\ldots,\mathbf{c}_{d}\in X^{[d]}$ such that for $i=1,\ldots d$:
\begin{enumerate}
\item $\mathbf{c}_{i}(\epsilon)=x_{1}$ for $\epsilon\in  F_{1}^0\cup\ldots\cup F_{i}^0$;
\item $\mathbf{c}_{i}(\epsilon)\in\RP^{[d-1]}[x],\epsilon\in \{0,1\}^d$;
\item $\mathbf{c}_1 \in\overline{\mathcal{O}(\mathbf{x},\mathcal{G}^{[d]})}$;
\item $\mathbf{c}_{i+1} \in\overline{\mathcal{O}(\mathbf{c}_i,\mathcal{G}^{[d]})}$.
\end{enumerate}
Assume this has been achieved, then we have
\[
\mathbf{c}_d \in\overline{\mathcal{O}(\mathbf{c}_{d-1},\mathcal{G}^{[d]})}\subset \ldots
\subset \overline{\mathcal{O}(\mathbf{c}_1,\mathcal{G}^{[d]})}
 \subset\overline{\mathcal{O}(\mathbf{x},\mathcal{G}^{[d]})}.
\]
As $\mathbf{c}_d({\epsilon})=x_1$ for all $\epsilon\in  F_{1}^0\cup\ldots\cup F_{d}^0$,
there exists some point $y\in X$
such that $\mathbf{c}_d(\epsilon)=x_1$ for $\epsilon\neq \vec{1}$
and $\mathbf{c}_d({\vec{1}})=y$.
Property (2) provides that $y\in \RP^{[d-1]}[x]$ which implies
$\mathbf{c}_d\in  \Q^{[d]}(X) $
by Theorem \ref{cube-minimal}.
Note that the system $(X,G)$ is distal, so is $(X^{[d]},\mathcal{G}^{[d]})$.
Thus $\mathbf{x}\in \overline{\mathcal{O}(\mathbf{c}_d,\mathcal{G}^{[d]})},$
which implies $\mathbf{x} \in \Q^{[d]}(X)$.

\medskip

We now return to the inductive constructions of $\mathbf{c}_{1},\ldots,\mathbf{c}_{d}\in X^{[d]}.$

For $i=1,\ldots,d,s=0,1$ and $\mathbf{y}=(y_\epsilon:\epsilon\in \{0,1\}^d)\in X^{[d]}$,
let
\[
\mathbf{y}|_{F_i^s}=(y_\epsilon:\epsilon\in \{0,1\}^d,\epsilon_i=s),
\]
then $\mathbf{y}|_{F_i^s}\in X^{[d-1]}$.
Define a map $\sigma_i $ from $G^{d-1}$ to $G^d$ such that
\[
(g_1,\ldots,g_{d-1})\mapsto (g_1,\ldots,g_{i-1},0,g_i,\ldots,g_{d-1}).
\]

Since $\mathbf{RP}^{[d-1]}(X)\subset \mathbf{RP}^{[d-2]}(X)$,
by inductive hypothesis we have $\mathbf{x}_{|F^0_{1}}\in \Q^{[d-1]}(X)$.
By Theorem \ref{cube-minimal} there exists some sequence
$\{\vec{g}_k^{1}\}_{k\in \N}\subset G^{d-1}$
 such that $\mathbf{g}_k^{(1)}=(\vec{g}_k^{1}\cdot\epsilon:\epsilon\in\{0,1\}^{d-1})$
 and
 \[
 \mathbf{g}_k^{(1)}\mathbf{x}_{|F^0_{1}}\to x_{1}^{[d-1]},\quad k\to \infty.
 \]

 For $k\in \N$,
 put $\mathbf{h}_k^{(1)}=(\sigma_1(\vec{g}_k^1)\cdot \epsilon:\epsilon\in \{0,1\}^d)$,
 then $\mathbf{h}_k^{(1)}\in \mathcal{F}^{[d]}$.
 Let $\mathbf{c}_1\in X^{[d]}$ be a limit point of
the sequence $\{\mathbf{h}_k^{(1)}\mathbf{x}\}_{k\in \N}$,
then $\mathbf{c}_{1}(\epsilon)=x_{1},\epsilon\in F^0_1$.
Notice that
 \[
 \mathbf{g}_k^{(1)}\mathbf{x}_{|F^1_{1}}\to \mathbf{c}_{1|F^1_{1}},\quad k\to \infty,
 \]
by equivalence of $\mathbf{RP}^{[d-1]}(X)$,
we get that $\mathbf{c}_1(\epsilon)\in \RP^{[d-1]}[x],\epsilon\in \{0,1\}^d$.

 Assume we have already constructed $\mathbf{c}_{1},\ldots,\mathbf{c}_{i}\in X^{[d]}$.
As $\mathbf{RP}^{[d-1]}(X)\subset \mathbf{RP}^{[d-2]}(X)$,
by inductive hypothesis $\mathbf{c}_{i|F^0_{i+1}}\in \Q^{[d-1]}(X)$.
By Theorem \ref{cube-minimal}
 there is some sequence
$\{\vec{g}_{k}^{i+1}\}_{k\in \N}\subset G^{d-1}$ such that
$\mathbf{g}_k^{(i+1)}=(\vec{g}_k^{i+1}\cdot\epsilon:\epsilon\in\{0,1\}^{d-1})$ and
\[
\mathbf{g}_k^{(i+1)}\mathbf{c}_{i|F^0_{i+1}}\to x_{1}^{[d-1]},\quad k\to\infty.
\]

For $k\in \N$,
 put $\mathbf{h}_k^{(i+1)}=(\sigma_{i+1}(\vec{g}_k^{i+1})\cdot \epsilon:\epsilon\in \{0,1\}^d)$,
 then $\mathbf{h}_k^{(i+1)}\in \mathcal{F}^{[d]}$.
 Let $\mathbf{c}_{i+1}\in X^{[d]}$ be a limit point of
the sequence $\{\mathbf{h}_k^{(i+1)}\mathbf{c}_i\}_{k\in \N}$, then $\mathbf{c}_{i+1}(\epsilon)=x_{1},\epsilon\in F^0_{i+1}$.
Moreover,
as $\mathbf{c}_{i}(\epsilon)=x_1,\epsilon\in F_1^0\cup \ldots\cup F_{i}^0$,
we get that $\mathbf{c}_{i+1}(\epsilon)=x_1,\epsilon\in F_1^0\cup \ldots \cup F_{i+1}^0$.
Notice that
$\mathbf{c}_i(\epsilon)\in \RP^{[d-1]}[x],\epsilon\in \{0,1\}^d$ and
\[
\mathbf{g}_k^{(i+1)}\mathbf{c}_{i|F^1_{i+1}}
\to \mathbf{c}_{i+1|F^1_{i+1}},\quad  k\to \infty,
\]
by equivalence of $\mathbf{RP}^{[d-1]}(X)$,
we deduce that
$\mathbf{c}_{i+1}(\epsilon)\in \RP^{[d-1]}[x],\epsilon\in \{0,1\}^d$
as was to be shown.

This completes the proof.
\end{proof}

\subsection{Nilpotent groups, nilmanifolds and nilsystems}
Let $L$ be a group.
For $g,h\in L$, we write $[g,h]=ghg^{-1}h^{-1}$ for the commutator of $g$ and $h$,
we write $[A,B]$ for the subgroup spanned by $\{[a,b]:a\in A,b\in B\}$.
The commutator subgroups $L_j,j\geq1$, are defined inductively by setting $L_1=L$
and $L_{j+1}=[L_j,L]$.
Let $k\geq 1$ be an integer.
We say that $L$ is \emph{k-step nilpotent} if $L_{k+1}$ is the trivial subgroup.

\medskip

Let $L$ be a $k$-step nilpotent Lie group and $\Gamma$ a discrete cocompact subgroup of $L$.
The compact manifold $X=L/\Gamma$ is called a \emph{k-step nilmanifold.}
The group $L$ acts on $X$ by left translations and we write this action as $(g,x)\to gx$.
Let $\tau\in L$ and $T$ be the transformation $x\to \tau x$ of $X$.
Then $(X,T)$ is called a \emph{k-step nilsystem}.

We also make use of inverse limits of nilsystems and so we recall the definition of an inverse limit
of systems (restricting ourselves to the case of sequential inverse limits).
If $\{(X_i,T_i)\}_{i\in \N}$ are systems with $\text{diam}(X_i)\leq 1$ and $\phi_i:X_{i+1}\to X_i$
are factor maps, the \emph{inverse limit} of the systems is defined to be the compact subset of $\prod_{i\in \N}X_i$
given by $\{(x_i)_{i\in \N}:\phi_i(x_{i+1})=x_i,i\in \N\}$,
which is denoted by $\lim\limits_{\longleftarrow}\{ X_i\}_{i\in \N}$.
It is a compact metric space endowed with the distance $\rho(x,y)=\sum_{i\in \N}1/ 2^i \rho_i(x_i,y_i)$.
We note that the maps $\{T_i\}$ induce a transformation $T$ on the inverse limit.

The following structure theorem characterizes inverse limits of nilsystems using dynamical cubespaces.

\begin{theorem}[Host, Kra and Maass]\cite[Theorem 1.2]{HKM}\label{description}
  Assume that $(X,T)$ is a minimal system
  and let $d\geq2$ be an integer. The following properties are equivalent:
  \begin{enumerate}
    \item If $\mathbf{x},\mathbf{y}\in \mathbf{Q}^{[d]}(X)$ have $2^d-1$ coordinates in common, then $\mathbf{x}=\mathbf{y}$.
    \item If $x,y\in X$ are such that $(x,y,\ldots,y)\in  \mathbf{Q}^{[d]}(X)$, then $x=y$.
    \item $X$ is an inverse limit of $(d-1)$-step minimal nilsystems.
  \end{enumerate}
\end{theorem}
This result shows that a minimal system $(X,T)$ is a system of order $d$
if and only if it is an inverse limit of minimal $d$-step nilsystems.


Thus for a minimal system $(X,T)$, we just call it a \emph{$d$-step pro-nilsystem} if $\mathbf{RP}^{[d]}(X)$ is trivial.
\begin{theorem}\cite[Theorem 3.6]{PD13}\label{system-of-order}
  A minimal system $(X,T)$
  is a system of order $\infty$ if and only if it is an inverse limit of minimal nilsystems.
\end{theorem}

\subsection{The Enveloping semigroups}
The \emph{enveloping semigroup} (or \emph{Ellis semigroup})
$E(X)$ of a dynamical system $(X,G)$ is defined as the closure in $X^X$ of the set $\{g :g\in G\}$
endowed with the product topology.
For an enveloping semigroup $E(X)$, the maps
$E(X)\to E(X),p \mapsto pq$ and $p\mapsto gp$ are continuous for all $g\in G,q\in E(X)$.

This notion was introduced by Ellis and has proved to be a useful tool in studying dynamical systems.
Algebraic properties of $E(X)$ can be translated into dynamical properties of $(X,G)$
and vice versa.
To illustrate this fact, we recall the following theorem.

\begin{theorem}\cite[Chapter 3,4 and 5]{JA}\label{distal-ellis-group}
  Let $(X,G)$ be a minimal system. Then
  \begin{enumerate}
    \item $E(X)$ is a group if and only if $(X,G) $ is distal.
    \item $E(X)$ is an abelian group if and only if $(X,G)$ is equicontinuous if and only if $E(X)$
    is a group of continuous transformations.
  \end{enumerate}
\end{theorem}

\section{Maximal factors of order $\infty$ of dynamical cubespaces}
Let $(X,G)$ be a minimal distal system.
The main goal of this section is to study
the maximal factor of order $\infty$
of the system $(\mathbf{Q}^{[l]}(X),\mathcal{G}^{[l]})$, where $l\in \N$.
It is shown that indeed such factor is
$(\mathbf{Q}^{[l]}(X_\infty),\mathcal{G}^{[l]})$,
where $X_{\infty}=X/\mathbf{RP}^{[\infty]}(X)$.

 We start with some simple observations.

\begin{lemma}\label{RP_d}
  Let $(X,G)$ be a minimal system and $d\in \N$.
  Then $(x,y)\in \mathbf{RP}^{[d]}(X)$ if and only if
  for each neighborhood $V$ of $y$, there exists some $\vec{g}\in G^{d+1}$
  such that $ (\vec{g}\cdot \epsilon)x\in V$ for all $\epsilon \in \{0,1\}^{d+1}\backslash \{\vec{0}\}$.
  \end{lemma}

\begin{proof}
  It follows from Theorem \ref{cube-minimal}.
\end{proof}

\begin{lemma}\label{cube-nil}
   Let $(X,G)$ be a system of order $d$, where $d\in \N\cup \{ \infty\}$.
   Then for every $l\in \N$,
   $(\Q^{[l]}(X),\mathcal{G}^{[l]})$ is also a system of order $d$.
\end{lemma}

\begin{proof}
Let $l\in \N$ and $d\in \N\cup\{\infty\}$.
For $\epsilon\in\{0,1\}^l$,
let $\pi_\epsilon:\Q^{[l]}(X)\to X$ be the projection of $\Q^{[l]}(X)$ on $\epsilon$-component.
Note that every element of $\mathcal{G}^{[l]}$ has form (\ref{rep-hk}).
  We consider the action of the group $\mathcal{G}^{[l]}$ on $\epsilon$-component via the representation
  \[
  (g+\vec{g}\cdot\epsilon:\epsilon \in \{0,1\}^l)
  \mapsto g+\vec{g}\cdot\epsilon.
  \]
  With respect to this action of $\mathcal{G}^{[l]}$ on $X$
  the map $\pi_\epsilon$ is a factor map $\pi_\epsilon:(\Q^{[l]}(X),\mathcal{G}^{[l]})\to (X,\mathcal{G}^{[l]}).$
  By Theorem \ref{lift-property} and Remark \ref{infi}, we have that
  \[
  (\pi_\epsilon\times \pi_\epsilon)\mathbf{RP}^{[d]}(\Q^{[l]}(X))=\mathbf{RP}^{[d]}(X,\mathcal{G}^{[l]})
  =\mathbf{RP}^{[d]}(X)=\Delta_X,
  \]
  as $X$ is a system of order $d$.
This shows that $\mathbf{RP}^{[d]}(\Q^{[l]}(X))$ is trivial and thus
$(\mathbf{Q}^{[l]}(X),\mathcal{G}^{[l]})$ is a system of order $d$.
\end{proof}

\begin{theorem}\label{infinity-step}
Let $(X,G)$ be a minimal distal system and let $ X_{\infty}=X/\mathbf{RP}^{[\infty]}(X).$
Then for every $l\in \N$,
the maximal factor of order $\infty$
of $(\mathbf{Q}^{[l]}(X),\mathcal{G}^{[l]})$ is $(\mathbf{Q}^{[l]}(X_\infty),\mathcal{G}^{[l]})$.
\end{theorem}

\begin{proof}
Let $l\in \N$ and let $\pi :X\to X_\infty$ be the factor map.

Note that $\pi^{[l]}:\Q^{[l]}(X)\to \Q^{[l]}(X_\infty)$ is a factor map,
where $\pi^{[l]}: X^{[l]}\rightarrow X_\infty^{[l]}$
is defined from $\pi $ coordinatewise.
It is sufficient to show that
\[
\mathbf{RP}^{[\infty]}(\mathbf{Q}^{[l]}(X))=R_{\pi^{[l]}}.
\]
By Lemma \ref{cube-nil}, $(\Q^{[l]}(X_\infty),\mathcal{G}^{[l]})$ is a system of order $\infty$.
By Remark \ref{infi} we have
 \[
 \mathbf{RP}^{[\infty]}(\mathbf{Q}^{[l]}(X))\subset R_{\pi^{[l]}}.
 \]

\medskip

It remains to show that if $(\mathbf{x},\mathbf{y})\in R_{\pi^{[l]}}$,
then $(\mathbf{x},\mathbf{y})\in \mathbf{RP}^{[\infty]}(\mathbf{Q}^{[l]}(X)).$
For convenience, we denote the regionally proximal of order $\infty$ pair $(\mathbf{x},\mathbf{y})$
in $\mathbf{Q}^{[l]}(X)$ by $\mathbf{x}\sim \mathbf{y}$
instead of $(\mathbf{x},\mathbf{y})\in \mathbf{RP}^{[\infty]}(\mathbf{Q}^{[l]}(X)).$

We need following claims.

\medskip

\noindent {\bf Claim 1}:
Let $\mathbf{x},\mathbf{y}\in \Q^{[l]}(X)$ with $\mathbf{x}\sim \mathbf{y}$
and $f$ be an Euclidean permutation, then
we still have $f(\mathbf{x})\sim f(\mathbf{y})$.

\begin{proof}[Proof of Claim 1]
Notice that the Euclidean permutation $f$ leaves $\mathbf{Q}^{[l]}(X)$
and $\mathcal{G}^{[l]}$ invariant.
We obtain Claim 1 by using this fact.
\end{proof}

Fix $x\in X$ and let
\[
\RP^{[\infty]}[x]=\{ y\in X:(x,y)\in  \mathbf{RP}^{[\infty]}(X)  \}.
\]

\medskip

\noindent {\bf Claim 2}:
Let $y\in\RP^{[\infty]}[x]$, then
$x^{[l]}\sim(x,y^{[l]}_*).$

\begin{proof}[Proof of Claim 2]
Let $W$ be a neighbourhood of $x$ and $k\in\N$. Put $n=(k+1)\cdot l$.

It follows from Theorem \ref{cube-minimal} that
 $(x,y^{[n]}_*)\in \Q^{[n]}(X)$,
and then there exists some $\vec{g}=(g_1,\ldots,g_n)\in G^n$ such that
$\mathbf{g}= (\vec{g}\cdot \epsilon:\epsilon\in \{0,1\}^n)$ and
\begin{equation}\label{infity}
\mathbf{g} (x,y^{[n]}_*)\in W^{ [n]}.
\end{equation}

For $i=1,\ldots, k+1$,
let $\vec{g}_i=(g_{(i-1)l+1},\ldots,g_{il})$
and $\mathbf{g}_i=(\vec{g}_i\cdot \epsilon:\epsilon\in\{0,1\}^l)\in \mathcal{F}^{[l]}$.
Then by (\ref{infity}), for every
$\omega\in \{0,1\}^{k+1}\backslash\{\vec{0}\}$ we have
\[
(\sum_{i=1}^{k+1}\mathbf{g}_i\cdot \omega_i)
 (x,y^{[l]}_*) \in W^{[l]},
\]
which implies
$(x^{[l]},(x,y^{[l]}_*))\in \RP^{[k]}(\mathbf{Q}^{[l]}(X))$
by Lemma \ref{RP_d}.
As $k$ is arbitrary, we conclude that $x^{[l]}\sim(x,y^{[l]}_*)$.
This shows Claim 2.
\end{proof}

\medskip

\noindent {\bf Claim 3}:
Let $\mathbf{y}=(y_\epsilon)\in X^{[l]}$ with $y_\epsilon\in \RP^{[\infty]}[x],\epsilon\in\{0,1\}^l$, then
$x^{[l]}\sim  \mathbf{y}.$

\begin{proof}[Proof of Claim 3]
Let $\mathbf{y}=(y_\epsilon)\in X^{[l]}$ with $y_\epsilon\in \RP^{[\infty]}[x],\epsilon\in\{0,1\}^l$.
First, by Proposition \ref{fiber-property} we obtain that $\mathbf{y}\in \Q^{[l]}(X)$.
Let
\[
N_{\mathbf{y}}=|\{  \epsilon\in\{0,1\}^l:x\neq y_\epsilon \}|,
\]
 then $0\leq N_{\mathbf{y}} \leq 2^l$.
  We show Claim 3 by induction on $N_{\mathbf{y}}$.

Assume $N_{\mathbf{y}}=0$, $\mathbf{y}=x^{[l]}$, the result is trivial.

Assume $N_{\mathbf{y}}=1$. Choose an Euclidean permutation $f$ such that
  $f(\mathbf{y})=(y,x^{[l]}_*)$ for some $y\in \RP^{[\infty]}[x]$.
  By Claim 2, $(y,x^{[l]}_*)\sim y^{[l]}$.
 Note that $x^{[l]}  \sim  y^{[l]}$ and so by equivalence
 of $\mathbf{RP}^{[\infty]}(\mathbf{Q}^{[l]}(X))$, we get
 $(y,x^{[l]}_*)\sim x^{[l]}$.
 Again by Claim 1, $x^{[l]}\sim  \mathbf{y}$.

Let integer $1<d\leq 2^l$.
  Suppose the statement is true for all $j=1,\ldots,d-1$.
  That is, if
  $\mathbf{z}=(z_\epsilon)\in X^{[l]}$ with $z_\epsilon\in \RP^{[\infty]}[x],\epsilon\in\{0,1\}^l$
  and $N_{\mathbf{z}}\leq d-1$, then we have $x^{[l]}\sim \mathbf{z}$.

  Let
  $\mathbf{y}=(y_\epsilon)\in X^{[l]}$ with $y_\epsilon\in \RP^{[\infty]}[x],\epsilon\in\{0,1\}^l$
  and $N_{\mathbf{y}}=d$.
 Choose $\eta\in \{0,1\}^l$ with $y_\eta \neq x$.
 Let $\mathbf{u}=(u_\epsilon)\in X^{[l]}$ such that $u_\eta=x$ and $u_\epsilon=y_\epsilon$ otherwise,
 then $N_{\mathbf{u}}=d-1$ and $\mathbf{u}\sim x^{[l]}$ by inductive hypothesis.

 On the other hand, as $\mathbf{y}\in \Q^{[l]}(X)$ and by minimality of the system $(\Q^{[l]}(X),\mathcal{G}^{[l]})$
 there is some sequence $\{\mathbf{g}_i\}_{i\in \N} \subset \mathcal{G}^{[l]}$
 such that $\mathbf{g}_i \mathbf{y}\to x^{[l]},i\to \infty$.
 Without loss of generality,
 assume that $\mathbf{g}_i \mathbf{u}\to \mathbf{u}',i\to \infty$ for some $\mathbf{u}'=(u'_\epsilon)\in X^{[l]}$.
Then $u_\eta'=u'$ for some $u'\in X$ and $u'_\epsilon=x$ otherwise.
By equivalence of $\mathbf{RP}^{[\infty]}(X)$,
we get $u'\in \RP^{[\infty]}[x]$.
So $N_{\mathbf{u'}}=1$ and by inductive hypothesis $\mathbf{u}'\sim x^{[l]}$.
The system $(X,G)$ is distal, so is $(\Q^{[l]}(X)\times \Q^{[l]}(X),\mathcal{G}^{[l]})$.
Thus the point $(\mathbf{u},\mathbf{y})$ also belongs to the orbit
closure of the point $(\mathbf{u}',x^{[l]})$
under $\mathcal{G}^{[l]}$-action.
By equivalence of $\mathbf{RP}^{[\infty]}(\Q^{[l]}(X))$,
we get that $\mathbf{u}\sim \mathbf{y}$ and $x^{[l]}\sim \mathbf{y}$.

This shows Claim 3.
\end{proof}

Now let $\mathbf{x}=(x_\epsilon),\mathbf{y}=(y_\epsilon)\in \Q^{[l]}(X)$
with $(\mathbf{x},\mathbf{y})\in R_{\pi^{[l]}}$.
That is,
$(x_\epsilon,y_\epsilon)\in \mathbf{RP}^{[\infty]}(X),\epsilon\in \{0,1\}^l$.
By minimality of the system $(\Q^{[l]}(X),\mathcal{G}^{[l]})$,
there is some sequence $\{\mathbf{g}_i\}_{i\in \N}\subset \mathcal{G}^{[l]}$
such that $\mathbf{g}_i\mathbf{x}\to x^{[l]},i\to  \infty$.
 Without loss of generality, assume that $\mathbf{g}_i\mathbf{y}\to \mathbf{z},i\to  \infty$
 for some $\mathbf{z}=(z_\epsilon)\in \Q^{[l]}(X)$.
 By equivalence of $\mathbf{RP}^{[\infty]}(X)$,
 we have $z_\epsilon\in \RP^{[\infty]}[x],$ $\epsilon\in\{0,1\}^l$ and so $\mathbf{z}\sim  x^{[l]}$ by Claim 3.
 Moreover we deduce that $\mathbf{x}\sim \mathbf{y}$,
 which means
$R_{\pi^{[l]}}\subset \mathbf{RP}^{[\infty]}(\mathbf{Q}^{[l]}(X))$
as was to be shown.

We conclude that the maximal factor of order $\infty$
of $\mathbf{Q}^{[l]}(X)$ is $\mathbf{Q}^{[l]}(X_\infty)$.
\end{proof}

\section{Host-Kra cubegroups of top-nilpotent Ellis groups}\label{NIL}
Let $E$ be a $d$-step top-nilpotent Ellis group, where $d\in \N$.
For $l\in \N$, endow $E^{[l]}$ with the product topology,
then $E^{[l]}$ is also a $d$-step top-nilpotent Ellis group.
The main goal of this section is to study
the closed subgroup generated by the Host-Kra cubegroup $\mathcal{HK}^{[l]}(E)$,
denoted by $\widetilde{E}^{[l]}$.
To do this, the binary cubegroups $\{\mathcal{C}_j^{[l]}(E)\}_{j=1}^d$ are introduced
 which are much easier to handle.
It is shown that the $j$-step topological commutator subgroup
of $\widetilde{E}^{[l]}$ is included in the group $\mathcal{C}_j^{[l]}(E)$
 for every $j=1,\ldots,d$.
We start by recalling the definition of Ellis groups.

\subsection{Ellis groups and quotients}
A \emph{right topological group} consists of a group $E$ and a topology $\mathcal{T}$ on $E$
such that the map $E\to E, q \mapsto qp$ is continuous for all $p\in E$.
A right topological group $E$ is called an \emph{Ellis group}
if $(E,\mathcal{T})$ is also a compact Hausdorff space.


In the sequel, we assume that $E$ is an Ellis group
and denote by $e$ the unit element.
We remark that
a subset $K$ of $E$ is closed (open)
if and only if $Kg$ is closed (open)
for all $g\in E$  by right-continuity.
 If $H$ is a closed subgroup of $E$,
 then $H$ with the subtopology induced from $\mathcal{T}$ is also an Ellis group.

\medskip

Let $H$ be a closed normal subgroup of $E$.
In this subsection,
we discuss when can the quotient group $E/H$
with the quotient topology be an Ellis group?

\begin{lemma}\label{quotient-group}
Let $E$ be an Ellis group and
let $H$ be a closed normal subgroup of $E$.
If $H$ satisfies that for any neighbourhood $U$ of $H$,
  there is a neighbourhood $V$ of $e$ such that $VH\subset U$,
  then the quotient group $E/H$ with the
  quotient topology is an Ellis group.
\end{lemma}

\begin{proof}
Since $H$ is normal in $E$,
we can define
the quotient group $\widetilde{E}=E/H$ and let $\pi:E\to \widetilde{E}$
   be the quotient map.
   Endow $\widetilde{E}$ with the quotient topology such that $\pi$ is continuous.
$\widetilde{E}$ with such topology is compact, as $E$ is compact and the map $\pi$ is continuous and onto.

Let $p,q\in E$ with $\pi(p)\neq \pi(q)$, then $pH\cap qH=\emptyset$.
As $H$ is normal, we have that $pH=Hp,qH=Hq$ and $Hp\cap Hq =\emptyset$.
By right-continuity, $Hp$ and $Hq$ are closed.
We can choose open neighbourhoods $U_1,U_2$ of $Hp,Hq$ respectively such that $U_1\cap U_2=\emptyset$,
as $E$ is a compact Hausdorff space.

Put $U=U_1p^{-1}\cap U_2 q^{-1}$,
by right-continuity $U$ is open and $H\subset U$.
By assumption,
there is a neighbourhood $V$ of $e$ such that
$VH\subset U$,
then $VpH\cap VqH=\emptyset$.
Notice that $\pi$ is open,
so $\pi(Vp),\pi(Vq)$ are neighbourhoods of $\pi(p),\pi(q)$ respectively in $\widetilde{E}$.
Moreover, we have that
$\pi(Vp)\cap \pi(Vq)=\emptyset$
which implies that
 $\widetilde{E}$ with the quotient topology is a Hausforff space.

Fix $q\in E$. Let $\{p_\alpha\}$ be a net in $E$
with $\pi(p_\alpha) \to \pi(p) $ for some $p\in E$.
Without loss of generality, we may assume that $p_\alpha \to p'\in E$,
 then $\pi(p')=\pi(p)$.
By right-continuity in $E$, we obtain that $p_\alpha q\to p' q$.
  As $\pi$ is continuous, we deduce that
  \[
\pi(p_\alpha)\pi(q)=\pi( p_\alpha q)\to \pi(p'q)=\pi(p')\pi(q)=\pi(p)\pi(q).
  \]
  This shows that the multiplication in $\widetilde{E}$ is right-continuous.

We conclude that the group $E/H$
with the quotient topology is an Ellis group.
\end{proof}

By the following classical results,
we can show that
if the closed subgroup $H$ is included in the center of the group $E$,
then it satisfies the condition in Lemma \ref{quotient-group}.

\begin{theorem}\label{joint-con}\cite[Appendix B, Theorem B.17]{JDV}
  Let $T$ be a group with a compact Hausdorff topology such that all right translations
  $t \mapsto ts: T\to T$ for $s\in T$ are continuous.
  Let $Y$ be a compact Hausdorff space and let $\pi: T\times Y\to Y$ be action of $T$ on $Y$
  which is separately continuous, i.e. each translation $\pi^{t}: Y\to Y$ and each motion
  $\pi_y :T\to Y$ is continuous $(t\in T,y\in Y)$.
  Then $\pi$ is jointly continuous.
\end{theorem}

\begin{cor}\label{top-group}\cite[Appendix B, Corollary B.18]{JDV}
  Let $G$ be a group with a compact Hausdorff topology such that the multiplication in $G$
  is separately continuous, i.e. all left translations $g\mapsto  hg$ and all right translations $h \mapsto hg$
  are continuous. Then $G$ is a topological group.
\end{cor}

\begin{lemma}\label{joint}
Let $E$ be an Ellis group and
let $H$ be a closed subgroup of $E$.
If $H$ is included in the center of the group $E$,
then the quotient
group $E/H$ with the quotient topology is an Ellis group.
\end{lemma}

\begin{proof}
 Note that the group $H$ with the subtopology 
 is also an Ellis group.
 Consider the map $\sigma$ from $H\times E$ to $E$ given by $(h,g) \mapsto  hg$.
 As $H$ is included in the center of the group $E$,
 we obtain that the map $\sigma$ is separately continuous.
 By Theorem \ref{joint-con} it is jointly continuous.

 Let $W\subset E$ be a neighbourhood of $H$.
For $h\in H$,
by joint-continuity of the map $\sigma$
there exist open subsets $U_h,V_h\subset E$ such that
 $h\in U_h, e\in V_h$ and
 \[
 V_h(U_h\cap H) =(U_h\cap H) V_h\subset W.
 \]
 As $H$ is compact and $\cup_{h\in H}(U_h\cap H)=H$, we can choose $n\in \N$ and $h_1,\ldots,h_n\in H$
 with $\cup_{i=1}^n(U_{h_i}\cap H)=H$.
 Put $V=\cap_{i=1}^n V_{h_i}$, then $V$ is open and $e\in V$.

Let $g\in H$, then $g\in U_{h_j}\cap H$ for some $j\in \{1,\ldots,n\}$
 and
 \[
 Vg\subset V(U_{h_j}\cap H)\subset V_{h_j}(U_{h_j}\cap H) \subset W.
 \]
This shows that $VH=\cup_{g\in H} Vg \subset W$.
By Lemma \ref{quotient-group},
 the quotient group $E/H$ with the quotient topology is an Ellis group.
\end{proof}

\subsection{Filtered Ellis groups and binary cubegroups}

A \emph{filtration} on an Ellis group $E$ is a descending sequence $\{ E_j\}_{j\in \N}$
of closed subgroups of $E$ such that $E=E_1\supset E_2\supset \ldots$
and for any $j\in \N$ we have $[E_j,E]\subset E_{j+1}$.\footnote{It is easy to check that every $E_i$ is a normal subgroup of $E$.}
We always assume that the filtration terminates at some point,
in the sense that for some $d$ we have $E_{d+1}=\{e\}$.
The least such $d$ is the \emph{length} of the filtration.
We just call $E$ a \emph{d-filtered Ellis group} if it is an Ellis group with a filtration of length $d$.

\begin{definition}\label{def-cube-group}
  Let $d\in \N$ and let $E$ be a $d$-filtered Ellis group.
  For integer $l\geq 2^d$ and $j=1,\ldots, d$,
let $\mathcal{C}_j^{[l]}(E)$ be the subgroup of $E^{[l]}$ spanned by
\begin{equation*}
\{g^{(F)}:g\in E_k, \;F\subset \{0,1\}^l ,\text{ codim$(F)=2^k$},\;k=j,\ldots, d \}.
\end{equation*}
The set
$\{\mathcal{C}_j^{[l]}(E)\}_{j=1}^d$ is called the \emph{binary cubegroups of dimension l} of $E$.
\end{definition}

Endow $E^{[l]}$ with the product topology.
It is clear that $E^{[l]}$ is also a $d$-filtered Ellis group.
In this subsection, we will show the following theorem.
\begin{theorem}\label{filtered-cube}
Let $d\in \N$ and let $E$ be a $d$-filtered Ellis group.
For every integer $l\geq 2^d$, we have
\begin{enumerate}
\item $[\mathcal{C}_j^{[l]}(E),\mathcal{C}_1^{[l]}(E)]\subset \mathcal{C}_{j+1}^{[l]}(E),j=1,\ldots,d-1$;
  \item $\mathcal{C}_j^{[l]}(E), j=1,\ldots, d$ are all closed in $E^{[l]}$.
\end{enumerate}
Moreover $\mathcal{C}_1^{[l]}(E)$ is a filtered Ellis group and the sequence
$\{\mathcal{C}_j^{[l]}(E)\}_{j=1}^d$ forms a filtration of length $d$ on $\mathcal{C}_1^{[l]}(E)$.
\end{theorem}

We first show Theorem \ref{filtered-cube} for $d=1$.

\begin{lemma}\label{case-abelian}
Let $G$ be an abelian Ellis group.
Let $l\geq2$ be an integer, for $d=1,\ldots,l-1$,
let $\mathcal{F}_d^{[
l]}$ be the
subgroup of $G^{[l]}$ spanned by
\[
\{g^{(F)}:g\in G, \;F\subset \{0,1\}^l , \;\mathrm{codim}(F)=d \}.
\]
Endow $G^{[l]}$ with the product topology,
then $\mathcal{F}_d^{[l]}$ is closed in $G^{[l]}$.
\end{lemma}

Before proving, we need say something about the elements of $\mathcal{F}_d^{[l]}$.
It is easy to see that for some elements of $\mathcal{F}_d^{[l]}$
there may be different representations.

For example, let $l=2,d=1$ and let
\[
F_i=\{\epsilon\in \{0,1\}^2:\epsilon_1=i\},
\quad S_i=\{\epsilon\in \{0,1\}^2:\epsilon_2=i\},\quad i=0,1,
\]
we have
$g^{[2]}=g^{\{0,1\}^2}=g^{(F_1)}\cdot g^{(F_2)}=g^{(S_1)}\cdot g^{(S_2)}$ for any $g\in G$.

Therefore,
we need rewrite all elements of $\mathcal{F}_d^{[l]}$ along a uniform order.
That is, there exist faces $F_i,i=1,\ldots,n$ of codimension $d$
where $n=n(l,d)\in \N$
such that
for any $\mathbf{g}\in \mathcal{F}_d^{[l]} $
there exist $g_1,\ldots,g_n\in G$ such that $\mathbf{g}=\prod_{i=1}^{n}g_i^{(F_i)}$
and $g_i$ is determined by $\mathbf{g},g_1,\ldots,g_{i-1},i=1,\ldots,n$,
where $g_0=e$.

\bigskip

\noindent {\bf Rewriting $\mathcal{F}_d^{[l]}$.}
Let $F\subset\{0,1\}^l$ be a face of codimension $d$ with
\[
  F=\{ \epsilon \in \{0,1\}^l: \epsilon_{i_1}=a_1,\ldots ,\epsilon_{i_d}=a_d \},
  \]
  where $1\leq {i_1}<\ldots<{i_d}\leq l$ and $a_i\in \{0,1\}$
such that if there exists some $j_0$ such that
$i_{j_0}>j_0$, then $i_j=0$ for all $j\geq j_0$.
Let $\Omega_d^l$ be the
collection of all such faces.

For a face $F\subset \{0,1\}^l$, let $\gamma^F$ be the maximal element of $F$,
which meaning
\[
\gamma^F\in F \text{ and } \gamma^F_i\geq \epsilon_i \text{ for all  } i=1,\ldots,l, \text{ and all  }\epsilon\in F.
\]
It is easy to see that
the maximal elements of different faces of $\Omega_d^l$ are different.

We define an order $\succ$ on $\{0,1\}^l$.
For different elements $\epsilon,\epsilon'\in\{0,1\}^l$,
put
\[
i=i(\epsilon,\epsilon')=\min\{ j=1,\ldots,l: \epsilon_j\neq \epsilon'_j\},
\]
we say
\[
\epsilon\succ\epsilon'\text{ if and only if } \epsilon_i> \epsilon'_i
.
\]

Following this, we
define an order still denoted by $\succ $ on $\Omega_d^l$.
For different faces $F,F'\in \Omega_d^l$, we say
\[
F\succ F' \text{ if and only if } \gamma^F\succ\gamma^{F'}
.
\]

Let
\[
\Omega_{d}^l=\{F_1\succ\ldots\succ F_{n}\},
\]
where $n=|\Omega_{d}^l|$.
Let $\gamma_k$ be the maximal element of $F_k$,
then $\gamma_k\notin \cup_{i=k+1}^{n} F_i$.

We claim that the faces of $\Omega_d^l$ meet the requirement.

\begin{prop}
  For any $F\in \Omega_{d}^l$, there exist $F^+,F^-\in \Omega_{d+1}^l$ such that
  \[
  F=F^+\cup F^-.
  \]
\end{prop}

\begin{proof}
Let $F\in \Omega_d^l$, then we have
  \[
  F=\{ \epsilon \in \{0,1\}^l: \epsilon_{i_1}=a_1,\ldots ,\epsilon_{i_d}=a_d \},
  \]
  where $1\leq {i_1}<\ldots<{i_d}\leq l$ and $a_i\in \{0,1\}$.

  If $i_d=d$,
  put
   \[
  F^+=\{ \epsilon \in \{0,1\}^l: \epsilon_{i_1}=a_1,\ldots ,\epsilon_{i_d}=a_d,\epsilon_{{d+1}}=1 \},
  \]
  and
  \[
  F^-=\{ \epsilon \in \{0,1\}^l: \epsilon_{i_1}=a_1,\ldots ,\epsilon_{i_d}=a_d,\epsilon_{{d+1}}=0 \},
  \]
  then we have $F^+,F^-\in \Omega_{d+1}^l$ and $F=F^+\cup F^-.$

  If $i_d>d$, let $j_0=\min\{ j=1,\ldots,d:i_j>j\}$.
  Put
   \[
  F^+=\{ \epsilon \in \{0,1\}^l: \epsilon_{i_1}=a_1,\ldots ,\epsilon_{i_d}=a_d,\epsilon_{i_{j_0}}=1 \},
  \]
  and
  \[
  F^-=\{ \epsilon \in \{0,1\}^l: \epsilon_{i_1}=a_1,\ldots ,\epsilon_{i_d}=a_d,\epsilon_{i_{j_0}}=0 \},
  \]
  then we have $F^+,F^-\in \Omega_{d+1}^l$ and $F=F^+\cup F^-.$
\end{proof}

Recall that we are considering  faces in $\Omega_d^l$.
\begin{prop}\label{reppppp}
For any $\mathbf{g}\in \mathcal{F}_d^{[l]}$,
there exist elements $g_1,\ldots,g_{n}\in G$ such that
\begin{equation}\label{formA}
\mathbf{g}=\prod_{i=1}^{n}g_i^{(F_i)}.
\end{equation}
Moreover, 
if there exist other elements $h_1,\ldots,h_{n}\in G$ such that
\[
\mathbf{g}=\prod_{i=1}^{n}h_i^{(F_i)},
\]
then $g_i=h_i,i=1,\ldots,n$.
\end{prop}

\begin{proof}

We first show that such decomposition exists.

Notice that since $G$ is abelian, it suffices to show that for
any $g\in G$ and any face $F$ of codimension $d$,
$g^{(F)}$ can be decomposed as form (\ref{formA}).
\medskip

\noindent {\bf Existence.}
For any $g\in G$ and any face $F$ of codimension $d$,
there exist $N\in \N$ and $F_k\in \Omega_d^l,g_k\in \{g,g^{-1}\},k=1,\ldots,N $
  such that
  \[
  g^{(F)}= \prod_{k=1}^N g_k^{(F_i)}.
  \]

 \begin{proof}[Proof of Existence]

We show the statement by induction on $d$.

Assume $d=1$.
Let $F$ be a face of codimension 1,
then $F$ can be written as
  \[
  F=\{ \epsilon \in \{0,1\}^l: \epsilon_{i}=a \},
  \]
  where $i\in \{1,\ldots,l\}$ and $a\in \{0,1\}$.

  If $a=0$ or $i=1$, then $F\in \Omega_1^l$.

  If $a=1$ and $i>1$. Let $S= \{\epsilon \in \{0,1\}^l: \epsilon_{i}=0 \},$ then $S\in \Omega_1^l$ and
  \[
  g^{(F)}=  g^{(F_0)}  g^{(F_1)}  (g^{-1})^{(S)},
  \]
  where $F_i= \{ \epsilon \in \{0,1\}^l: \epsilon_{1}=i \},i=0,1$.
  This shows the statement for $d=1$.

\medskip

  Let $d>1$ be an integer and assume that the statement is true for all $i=1,\ldots,d-1$.
 Let $F$ be a face of codimension $d$,
then $F$ can be written as
\[
  F=\{ \epsilon \in \{0,1\}^l: \epsilon_{i_1}=a_1,\ldots ,\epsilon_{i_d}=a_d \},
  \]
  where $1\leq {i_1}<\ldots<{i_d}\leq l$ and $a_i\in \{0,1\}$.
Put
  \[
  F^U=\{ \epsilon \in \{0,1\}^l: \epsilon_{i_1}=a_1,\ldots ,\epsilon_{i_{d-1}}=a_{d-1} \},
  \]
  then $F^U$ is a face of codimension $d-1$.
  By inductive hypothesis, there exist $N\in \N$ and
$F_k\in \Omega_{d-1}^l,g_k\in \{g,g^{-1}\},k=1,\ldots,N$ such that
  \begin{equation}\label{re}
 g^{(F^U)}=\prod_{k=1}^N g_k^{(F_k)}.
  \end{equation}
  For every $ k$, $F_k$ can be written as
  \[
  F_k=\{ \epsilon \in \{0,1\}^l: \epsilon_{i_1^{(k)}}=a_1^{(k)},
  \ldots ,\epsilon_{i_{d-1}^{(k)}}=a_{d-1}^{(k)} \},
\]
  where $1\leq {i_1^{(k)}}<\ldots<{i_{d-1}^{(k)}}\leq l$ and $a_i^{(k)}\in \{0,1\}$.
  By induction, we may assume additionally that ${i_{d-1}^{(k)}}\leq i_{d-1}$.
Put
 \[
  F_k^{(d)}=\{ \epsilon \in \{0,1\}^l: \epsilon_{i_1^{(k)}}=a_1^{(k)},
  \ldots ,\epsilon_{i_{d-1}^{(k)}}=a_{d-1}^{(k)} ,\epsilon_{i_d}=a_d\},
\]
then by (\ref{re})
we have
  \[
  g^{(F)}=\prod_{k=1}^Ng_k^{(F_k^{(d)})}.
  \]
If $i_d=d$ or $a_d=0$, then $F_k^{(d)}\in \Omega_{d}^l$ for every $k$.\\
If $i_d>d$ and $a_d=1$.
As $F_k\in \Omega_{d-1}^l$, 
we can choose $F_k^{(d)+},F_k^{(d)-}\in \Omega_{d}^l$ with
\[
F_k^{(d)+}\cup F_k^{(d)-}=F_k.
\] Put
  \[
  S_k^{(d)}=\{ \epsilon \in \{0,1\}^l: \epsilon_{i_1^{(k)}}=a_1^{(k)},
  \ldots ,\epsilon_{i_{d-1}^{(k)}}=a_{d-1}^{(k)} ,\epsilon_{i_d}=0\},
\]
then $ S_k^{(d)}\in  \Omega_{d}^l$ and
\[
g_k^{(F_k^{(d)})}=g_k^{(F_k^{(d)+})} g_k^{(F_k^{(d)-})}(g_k^{-1})^{(  S_k^{(d)})}.
\]
Hence we have
\[
 g^{(F)}=\prod_{k=1}^Ng_k^{(F_k^{(d)})} =\prod_{k=1}^N g_k^{(F_k^{(d)+})}
  g_k^{(F_k^{(d)-})}(g_k^{-1})^{(  S_k^{(d)})}
\]
as was to be shown.
 \end{proof}

Let $\mathbf{g}\in \mathcal{F}_d^{[l]}$.
Assume there exist elements $g_k,h_k\in G,k=1,\ldots,{n}$ such that
\[
\prod_{k=1}^{n}g_k^{(F_k)}=\mathbf{g}=\prod_{k=1}^{n}h_k^{(F_k)},
\]
we claim that $g_k=h_k$ for every $k$.

Indeed,
for every $k$
as $\gamma_k\notin \cup_{i=k+1}^{n} F_i$,
we get that
\[
\prod_{\substack{  1\leq j\leq k\\  \gamma_{k} \in F_j \;}} g_j=
\mathbf{g}({\gamma_k})=
\prod_{\substack{  1\leq j\leq k\\  \gamma_{k} \in F_j \;}} h_j.
\]
By induction
we deduce that $g_k=h_k$ for every $k$.
\end{proof}

Now let us show Lemma \ref{case-abelian}.

\begin{proof}[Proof of Lemma \ref{case-abelian}]
Let $\{\mathbf{g}_\alpha\}$ be a net in $\mathcal{F}_d^{[l]}$ with $\mathbf{g}_\alpha\to \mathbf{g}$
for some $\mathbf{g}\in G^{[l]}$.
That is,
$\mathbf{g}_\alpha(\epsilon)\to \mathbf{g}(\epsilon)$ for every $\epsilon\in \{0,1\}^l.$
We will show that $\mathbf{g}\in \mathcal{F}_d^{[l]}$.

By Proposition \ref{reppppp}, for every $\alpha$
there exist $g_{\alpha,i}\in G,i=1,\ldots, n$ such that
\[
\mathbf{g}_{\alpha}=\prod_{i=1}^{n}g_{\alpha,i}^{(F_i)}.
\]
Without loss of generality, by taking subnet we may assume that
$g_{\alpha,i}\to g_i\in G,i=1,\ldots,n$.
As $G$ is an abelian Ellis group,
by Corollary \ref{top-group} it is also a topological group.
From this,
we deduce that
\[
\mathbf{g}_{\alpha}=\prod_{i=1}^{n}g_{\alpha,i}^{(F_i)}\to
 \prod_{i=1}^{n}g_i^{(F_i)}
\]
and thus $\mathbf{g}=\prod_{i=1}^{n}g_i^{(F_i)}\in \mathcal{F}_d^{[l]}$
as was to be shown.
\end{proof}

Following the ideas in the proof of Lemma \ref{case-abelian},
we are able to show Theorem \ref{filtered-cube}.
Before giving the proof, we need some preparations.
\begin{lemma}\label{normal-center-series}
 Let $d\geq 2$ be an integer and let $E$ be a $d$-filtered Ellis group.
 Then for every $j=2,\ldots,d$,
 the quotient group $E/E_j$ with the quotient topology is a $(j-1)$-filtered Ellis group.
\end{lemma}

\begin{proof}
Let $d\geq 2$ be an integer.
Note that $E_j$ is normal in $E$ for every $ j$,
we can define the quotient group $E/E_j$.
By the argument in Lemma \ref{quotient-group},
it suffices to show that
the quotient group $E/E_j$ with the quotient topology is a Hausdorff space.

We show it by induction on $j$.

As $[E_d,E]\subset E_{d+1}=\{e\}$, we get that
$E_d$ is included in the center of $E$.
When $j=d$, it follows from
Lemma \ref{joint} immediately.
 %

Let integer $1\leq j\leq d-1$.
Suppose the statement is true for all $i=j+1,\ldots , d$,
then the quotient group
 $\widetilde{E}=E/E_{j+1}$ with the quotient topology induced from $E$ is an Ellis group.
Let $\phi_1:E\to \widetilde{E}$ be the quotient map,
then $\phi_1$ is continuous.
As $\phi_1$ is a continuous map from a compact space to a Hausdorff space,
it is closed.
Thus $\widetilde{E_j}=\phi_1(E_j)$ is a closed subset of $\widetilde{E}$.
By the condition $[E,E_j]\subset E_{j+1}$,
$\widetilde{E_j}$ is included in the center of $\widetilde{E}$.
Notice that $\widetilde{E_j}$ is also a group.
By Lemma \ref{joint}, the quotient group
$\widetilde{E}/\widetilde{E_j}$ with the quotient topology induced from $\widetilde{E}$ is
an Ellis group.
 Let $\phi_2 : \widetilde{E}\to  \widetilde{E}/\widetilde{E_j}$ be the quotient map,
 then $\phi_2$ is continuous,
 and so is the map $\phi=\phi_2\circ \phi_1:E\to \widetilde{E}/\widetilde{E_j}$.

On the other hand,
 the maps $\phi_1,\phi_2$ are also group homomorphisms,
so is $\phi$.
We obtain that the quotient group $E/\ker(\phi)$ is isomorphic to the group $\widetilde{E}/\widetilde{E_j}$
by group isomorphism theorem.
Clearly $\ker(\phi)=E_j$, so the quotient group
$E/E_j$ is isomorphic to $\widetilde{E}/\widetilde{E_j}$ and we denote this map by $f$.
Then $f$ is one to one.
Endow $E/E_j$ with the quotient topology induced from $E$
and let $\pi:E\to E/E_j$ be the quotient map,
then $\phi=f \circ \pi$.
We claim that $f$ is continuous.
Indeed, let $V$ be an open subset of $\widetilde{E}/\widetilde{E_j}$,
then $\pi^{-1}(f^{-1}(V))=\phi^{-1}(V)$ is open.
By the definition of the quotient topology,
$f^{-1}(V)$ is open in $E/E_j$.
This shows that $f$ is continuous.
Moreover,
$f$ is homeomorphic.
As the group $\widetilde{E}/\widetilde{E_j}$ is a Hausdorff space,
so is $E/E_j$.

By the argument above,
it is easy to see that
\[
E/E_j= E_1/E_j\supset E_2/E_j\supset \ldots\supset E_{j-1}/E_j\supset E_j/E_j=\{\tilde{e}\}
\]
is a filtration of length $j-1$ on $E/E_j$.

 This completes the proof.
\end{proof}

 As a consequence of Lemma \ref{normal-center-series}, we can easily get the following lemma.
\begin{lemma}\label{product}
 Let $E$ be a filtered Ellis group.
 Let $\{ p_\alpha\}$ be a net in $E$ and let $\{q_\alpha\}$ be a net in $E_j$
such that $p_\alpha \to p$ and $q_\alpha\to q$, where $p\in E,q\in E_j$.
Then every limit point of the net $\{ p_\alpha q_\alpha\}$ belongs to $pqE_{j+1}$.

\end{lemma}

\begin{proof}
Let $j\in \N$.
We may assume that the group $E_j$ is nontrivial,
otherwise there is nothing to prove.
Let $u\in E$ be some limit point of the net $\{ p_\alpha q_\alpha\}$.
Without loss of generality, we may assume that $p_\alpha q_\alpha \to u$.
  By Lemma \ref{normal-center-series}, the quotient group $E/E_{j+1}$ with the quotient topology
  is an Ellis group.
  Let $\pi:E\to E/E_{j+1}$ be the quotient map,
  then $\pi$ is continuous.
  Also $E_j/E_{j+1}=\pi(E_j)$ is a closed subgroup of $E/E_{j+1}$ which is
 included in the center of $E/E_{j+1}$.

Let $\sigma$ be the map from $E_j/E_{j+1}\times E/E_{j+1}$
to $E/E_{j+1}$ given by $(\pi(h),\pi(g)) \mapsto \sigma(\pi(h),\pi(g) )= \pi(hg)$.
   Clearly the map $\sigma$ is separately continuous.
 By Theorem \ref{joint-con} it is jointly continuous.
It follows that
\[
\pi(p_\alpha q_\alpha)=\pi(p_\alpha) \pi(q_\alpha)\to \pi(p)\pi(q)=\pi(pq)
\]
 which implies that $u= pqr$ for some $r\in E_{j+1}$.

  This completes the proof.
\end{proof}

Now we give a proof of Theorem \ref{filtered-cube}.

\begin{proof}[Proof of Theorem \ref{filtered-cube}]

Let $d,l\in \N$ with $l\geq 2^d$ and let $E$ be an Ellis group with a filtration
$E=E_1\supset E_2\supset\ldots\supset E_d\supset E_{d+1}=\{e\}.$
Let $\mathcal{C}_{d+1}^{[l]}(E)=\{e^{[l]}\}$.

We claim that for $j=1,\ldots,d$,
$[\mathcal{C}_j^{[l]}(E),\mathcal{C}_1^{[l]}(E)]\subset \mathcal{C}_{j+1}^{[l]}(E)$
and $\mathcal{C}_j^{[l]}(E)$ is closed in $E^{[l]}$.
We show it by induction on $d$.

When $d=1$.
$E$ is an abelian group, so is $E^{[l]}$,
we have $[\mathcal{C}_1^{[l]}(E),\mathcal{C}_1^{[l]}(E)]=\{e^{[l]} \}$.
It follows Lemma \ref{case-abelian} that $\mathcal{C}_1^{[l]}(E)$ is closed in $E^{[l]}$.

Let $d\geq 2$ be an integer and and suppose the statement is true for all $i=1,\ldots,d-1$.
As $\mathcal{C}_d^{[l]}(E)$ is included in the center of the group
$\mathcal{C}_1^{[l]}(E)$, we have that
$[\mathcal{C}_d^{[l]}(E),\mathcal{C}_1^{[l]}(E)]=\{e^{[l]}\}=\mathcal{C}_{d+1}^{[l]}(E)$.
It follows from Lemma \ref{case-abelian} that $\mathcal{C}_d^{[l]}(E)$ is closed in $E_d^{[l]}$
and thus it is closed in $E^{[l]}$.

Let $j<d$ be an integer and assume that we have already show that for $i=j+1,\ldots,d,$
$[\mathcal{C}_i^{[l]}(E),\mathcal{C}_1^{[l]}(E)]\subset \mathcal{C}_{i+1}^{[l]}(E)$
and $\mathcal{C}_i^{[l]}(E)$ is closed.

Let $\mathbf{g}\in \mathcal{C}_1^{[l]}(E)$ and $\mathbf{h}\in \mathcal{C}_{j+1}^{[l]}(E)$,
then $\mathbf{g}^{-1}\mathbf{h}\mathbf{g}=\mathbf{h}[\mathbf{h}^{-1},\mathbf{g}^{-1}]\in \mathcal{C}_{j+1}^{[l]}(E)$.
This implies that $\mathcal{C}_{j+1}^{[l]}(E)$ is a normal subgroup of $\mathcal{C}_{1}^{[l]}(E)$.
We deduce that
$\mathcal{C}_{i}^{[l]}(E)/\mathcal{C}_{j+1}^{[l]}(E),i=1,\ldots,j$ is the group
spanned by
\[
\{g^{(F)}\mathcal{C}_{j+1}^{[l]}(E):g\in E_k, \;F\subset \{0,1\}^l ,\text{ codim$(F)=2^k$},\;k=i,\ldots, j \}.
\]

For any $g\in E_j,h\in E_s$ and any faces
$F,S$
with codim$(F)=2^{j}$, codim$(S)= 2^{s}$ where $s\in \{1,\ldots,j\}$,
we have
\[
[g^{(F)},h^{(S)}]=[g,h]^{(F\cap S)}.
\]
Notice that
$\text{codim}(F\cap S)\leq \text{codim}(F)+\text{codim}( S)\leq2^{j+1}
$
and $[g,h]\in[E_{j},E]\subset E_{j+1}$, we get that
$[g^{(F)},h^{(S)}]=[g,h]^{(F\cap S)}\in \mathcal{C}_{j+1}^{[l]}(E)$.

This shows that
$\mathcal{C}_{j}^{[l]}(E)/\mathcal{C}_{j+1}^{[l]}(E)$ is included in the
center of the group $\mathcal{C}_{1}^{[l]}(E)/\mathcal{C}_{j+1}^{[l]}(E)$
and thus $[\mathcal{C}_j^{[l]}(E),\mathcal{C}_1^{[l]}(E)]\subset \mathcal{C}_{j+1}^{[l]}(E)$.
We next show that $\mathcal{C}_j^{[l]}(E)$ is closed.

Let $G=E/E_d$.
It follows from Lemma \ref{normal-center-series}
that $G$ is a $(d-1)$-filtered Ellis group.
By inductive hypothesis,
$\mathcal{C}^{[l]}_{i}(G),i=1,\ldots,d-1$ are all closed in $G^{[l]}$.
Let $\pi: E \to G$ be the quotient map.
The map $\pi^{[l]}: E^{[l]}\to G^{[l]}$
is defined from $\pi $ coordinatewise. Then $\pi^{[l]}$ is continuous.
As $\mathcal{C}^{[l]}_{j}(G)$ is closed in $G^{[l]}$,
we get that
   \[
   (\pi^{[l]})^{-1} \mathcal{C}^{[l]}_{j}(G)
= \mathcal{C}^{[l]}_{j}(E)\cdot E_d^{[l]}
\]
is closed in $E^{[l]}$.
Notice that $E_d$ is included in the center of the group $E$,
then $\mathcal{C}^{[l]}_{j}(E)\cdot E_d^{[l]}$ is a group.
Now the group $\mathcal{C}^{[l]}_{j}(E)\cdot E_d^{[l]}$ with the subtopology
induced from $E^{[l]}$ is an Ellis group.
Moreover,
for $i=j+1,\ldots,d$, we have
\[
[\mathcal{C}^{[l]}_{i}(E),\mathcal{C}^{[l]}_{j}(E)\cdot E_d^{[l]}]
\subset [\mathcal{C}^{[l]}_{i}(E),\mathcal{C}^{[l]}_{1}(E)]
\subset \mathcal{C}^{[l]}_{i+1}(E).
\]
This shows that
\[
\mathcal{C}^{[l]}_{j}(E)\cdot E_d^{[l]}\supset \mathcal{C}^{[l]}_{j+1}(E)\supset \ldots \supset
 \mathcal{C}^{[l]}_{d}(E)\supset \{e^{[l]}\}
 \]
 is a filtration on $\mathcal{C}^{[l]}_{j}(E)\cdot E_d^{[l]}$.
 By Lemma \ref{normal-center-series},
the quotient group
\[
W=\mathcal{C}^{[l]}_{j}(E)\cdot E_d^{[l]}/\mathcal{C}^{[l]}_{j+1}(E)
\]
with the quotient topology is an Ellis group. Moreover,
it is an abelian topological group by Corollary \ref{top-group}.

Let
\[
\Omega_{2^j}^l=\{F_1\succ\ldots\succ F_{n}\},
\]
where $n=|\Omega_{2^j}^l|$.
By the argument in Lemma \ref{case-abelian},
we deduce that for any $\mathbf{h}\in \mathcal{C}_j^{[l]}(E)$,
there exist $h_1,\ldots,h_n\in E_j$ such that
\begin{equation}\label{QQQQ}
\mathbf{h}\mathcal{C}_{j+1}^{[l]}(E)=\prod_{i=1}^{n}h_i^{(F_i)}\mathcal{C}_{j+1}^{[l]}(E).
\end{equation}

Let $\{\mathbf{g}_\alpha\}$ be a net in $\mathcal{C}_j^{[l]}(E)$ with
$\mathbf{g}_\alpha\to \mathbf{g}$, where $\mathbf{g}\in E^{[l]}$.
As $\mathcal{C}_j^{[l]}(E)\cdot E_d^{[l]}$ is closed in $E^{[l]}$ and it also contains
$\mathcal{C}_j^{[l]}(E)$, we get that $\mathbf{g}\in \mathcal{C}_j^{[l]}(E)\cdot E_d^{[l]}$
and
\[
\mathbf{g}_\alpha\mathcal{C}_{j+1}^{[l]}(E)\to \mathbf{g}\mathcal{C}_{j+1}^{[l]}(E) \quad \text{in }W.
\]
By (\ref{QQQQ}),
for every $\alpha$ there exist $g_{\alpha,i}\in E_j,i=1,\ldots,n$
such that
\[
\mathbf{g}_{\alpha}\mathcal{C}_{j+1}^{[l]}(E)=\prod_{i=1}^{n}
g_{\alpha,i}^{(F_i)}\mathcal{C}_{j+1}^{[l]}(E).
\]
Without loss of generality, by taking subnet we may assume that
\[
g_{\alpha,i}\to g_i,\quad \text{in  } E_j,
\]
for $g_i\in E_j,i=1,\ldots,n$.
Then for every $i$,
\[
g_{\alpha,i}^{(F_i)}\to g_i^{(F_i)}\quad \text{in  } \mathcal{C}_j^{[l]}(E)\cdot E_d^{[l]}.
\]
As $W$ is a topological group,
by considering the projection on $W$,
we have
\[
\mathbf{g}_{\alpha}\mathcal{C}_{j+1}^{[l]}(E)=\prod_{i=1}^{n}
g_{\alpha,i}^{(F_i)}\mathcal{C}_{j+1}^{[l]}(E)
\to \prod_{i=1}^{n}g_{i}^{(F_i)}\mathcal{C}_{j+1}^{[l]}(E)\quad \text{in  } W,
\]
which implies that
\[
\mathbf{g}\mathcal{C}_{j+1}^{[l]}(E)=\prod_{i=1}^{n}g_{i}^{(F_i)}\mathcal{C}_{j+1}^{[l]}(E).
\]
From this, we deduce that $\mathbf{g}\in\mathcal{C}_j^{[l]}(E) $
and thus $\mathcal{C}_j^{[l]}(E)$ is closed.

By induction we conclude that
$[\mathcal{C}_j^{[l]}(E),\mathcal{C}_1^{[l]}(E)]\subset \mathcal{C}_{j+1}^{[l]}(E)$
and $\mathcal{C}_j^{[l]}(E)$ is closed in $E^{[l]}$ for $j=1,\ldots,d$.

This completes the proof.
\end{proof}

\subsection{Top-nilpotent Ellis groups}

 Let $E$ be an Ellis group.
  For $A,B \subset E$, define $[A,B]_{\text{top}}$
as the closure of the subgroup $[A,B]$ spanned by $\{[a,b]:a\in A,b\in B\}$.
The topological commutators subgroups $E_j^{\text{top}},j\geq 1$,
are defined by setting $E_1^{\text{top}}=E$ and $E_{j+1}^{\text{top}}=[E_j^{\text{top}},E]_{\text{top}}$.
It is easy to see that $  E_{j}^{\text{top}} \supset E_{j+1}^{\text{top}}, j\geq 1.$
Let $d\in \N$.
We say that $E$ is
\emph{d-step top-nilpotent} if $E_{d+1}^{\text{top}}$ is the trivial subgroup.

\begin{lemma}\label{subgroup}
Let $E$ be an Ellis group, then
 $E_d^{\mathrm{top}}$
is a subgroup of $E$ for all $d\in \N$.
\end{lemma}

\begin{proof}
$E_1^{\mathrm{top}}=E$ is a group.
Let $d\geq2 $ be an integer
 and let $p,q\in E_d^{\text{top}}$.

 We will show that $pq,p^{-1} \in E_d^{\text{top}}$ and thus $E_d^{\text{top}}$ is a group.

Note that $E_d^{\text{top}}$ is the closure of the subgroup $[E_{d-1}^{\text{top}},E]$,
let $\{ p_\alpha\},\{q_\beta\}$ be nets in $[E_{d-1}^{\text{top}},E]$ with $p_\alpha\to p$ and $q_\beta\to q$.
For $r\in[E_{d-1}^{\text{top}},E]$, by right-continuity we have $q_\beta r\to q r$,
  which implies $q r\in E_d^{\text{top}}$.
 Since for every $\alpha$,
 $p_\alpha[p_\alpha^{-1},q^{-1}]\in [E_{d-1}^{\text{top}},E]$
and $p_\alpha q=q(p_\alpha [p_\alpha^{-1},q^{-1}])$,
we get that $p_\alpha q\in  E_d^{\text{top}}$.
By taking limit, we obtain $pq\in E_d^{\text{top}}$.

  Put $H= E_d^{\text{top}} p$.
  Then $H\subset  E_d^{\text{top}}$ and $H$ is closed.
  Let $B$ be a minimal closed subset of $H$ satisfying $B\cdot B\subset B$.
Such  $B$ exists by Zorn's lemma. Now
if $g\in B$, then $Bg$ is again closed and $Bg \cdot Bg \subset  Bg$. So $Bg=B$. Therefore
there exists $u\in B$ with $ug=g$. Then $u=e$ and $e\in H$,
which implies $p^{-1}\in E_d^{\text{top}}$.

We conclude that $E_d^{\text{top}}$
is a subgroup of $E$ for all $d\in \N$.
\end{proof}

We give a summary of what we will need in future to end this section.

\begin{prop}\label{cube-group-closed}
 Let $d,l\in \N$ with $l\geq 2^d$.
Let $E$ be a $d$-step top-nilpotent Ellis group.
For $j=1,\ldots,d$,
let $\mathcal{C}_j^{[l]}(E)$ be the subgroup of $E^{[l]}$ spanned by
\[
\{g^{(F)}:g\in E_i^{\text{top}}, F\subset \{0,1\}^l, \text{ codim$(F)=2^i$},\; i=j,\ldots, d \}.
\]
Then we have the following properties:

(1) $[\mathcal{C}_j^{[l]}(E),\mathcal{C}_1^{[l]}(E)]\subset \mathcal{C}_{j+1}^{[l]}(E),j=1,\ldots,d-1$.

(2) Endow $E^{[l]}$ with the product topology, then
$\mathcal{C}_j^{[l]}(E),j=1,\ldots,d$, are all closed.

(3) Let $\{p_\alpha\}$ be a net in $E$ and $\{q_\alpha\}$ be a net in $E_j^{\mathrm{top}}$
such that $p_\alpha \to p$ and $q_\alpha\to q$,
then every limit point of the net $\{ p_\alpha q_\alpha\}$ belongs to $pqE_{j+1}^{\mathrm{top}}$.

(4)  Let $\mathbf{g}=(g_\epsilon:\epsilon\in \{0,1\}^l)\in \mathcal{C}_d^{[l]}(E)$, then
 \[
\prod_{\epsilon\in \{0,1\}^l}g_\epsilon^{(-1)^{|\epsilon|}}=e,
\]
where $|\epsilon|$ denotes the sum $\sum_{i=1}^{l}\epsilon_i$ for $\epsilon=(\epsilon_1,\ldots,\epsilon_l)$.

(5) Let $\widetilde{E}^{[l]}$ be the closed subgroup generated by the Host-Kra cubegroup
 $\mathcal{HK}^{[l]}(E)$,
 then $\widetilde{E}^{[l]}$ is also $d$-step top-nilpotent and for every $j=1,\ldots,d$,
 we have
 \[
 (\widetilde{E}^{[l]})_j^{\text{top}}\subset \mathcal{C}_j^{[l]}(E).
 \]
\end{prop}

\begin{proof}

By Lemma \ref{subgroup}
we get that for a $d$-step top-nilpotent Ellis group
$E$, its
topological commutators subgroups form a filtration.
That is,
\[
E=E_1^{\text{top}}\supset
E_{2}^{\text{top}}\supset \ldots \supset
E_{d}^{\text{top}} \supset E_{d+1}^{\text{top}}=\{e\}.
\]

So,
properties (1) and (2) follow from Theorem \ref{filtered-cube},
and property (3) follows from Lemma \ref{product}.
It remains to show properties (4) and (5).

Let $\mathbf{g}=(g_\epsilon:\epsilon\in \{0,1\}^l)\in \mathcal{C}_d^{[l]}(E)$, then
there exist $n\in \N$ and $g_i\in E_d^{\text{top}}$,
faces $F_i\subset \{0,1\}^l$ of codimension $2^d,i=1,\ldots,n$
such that
\[
\mathbf{g}=\prod_{i=1}^{n}g_i^{(F_i)}.
\]
Notice that $E_d^{\text{top}}$ is included in the center of the group of $E$,
thus
 \[
\prod_{\epsilon\in \{0,1\}^l}g_\epsilon^{(-1)^{|\epsilon|}}=
\prod_{\epsilon\in \{0,1\}^l}\prod_{\epsilon\in F_i,1\leq i\leq n}g_i^{(-1)^{|\epsilon|}}=
\prod_{i=1}^{n}
\prod_{\epsilon\in F_i}g_i^{(-1)^{|\epsilon|}}=e.
\]
This shows property (4).

Recall that the Host-Kra cubegroup $\mathcal{HK}^{[l]}(E)$ is the subgroup spanned by
\[
\{g^{(F)}:g\in E, \;F\subset \{0,1\}^l, \text{ codim$(F)=1$} \},
\]
then $\mathcal{HK}^{[l]}(E)\subset \mathcal{C}_1^{[l]}(E)$.
By property (2), $\mathcal{C}_1^{[l]}(E)$ is a closed subgroup and thus we have that
$\widetilde{E}^{[l]} \subset\mathcal{C}_1^{[l]}(E)$.
By induction and properties (1) and (2), we can easily get
\[
 (\widetilde{E}^{[l]})_j^{\text{top}}\subset \mathcal{C}_j^{[l]}(E),\quad j=1,\ldots,d.
 \]

 This completes the proof.
\end{proof}

\section{Minimal systems with top-nilpotent enveloping semigroups}

In this section, we discuss minimal systems
 with top-nilpotent enveloping semigroups and show that such systems
are indeed systems of order $\infty$.

For a distal system $(X,G)$,
let $\{  E_d^{\text{top}}(X)\}_{d\in \N}$ denote the sequence of topological commutators of $E(X)$.
Let $\pi :X\to Y$ be the factor map between systems $(X,G)$ and $(Y,G)$.
There is a unique continuous semigroup homomorphism $\pi^*:E(X)\to E(Y)$
such that $\pi(ux)=\pi^*(u)\pi(x)$ for all $x\in X$ and $u\in E(X)$.

Note that if $\pi :X\to Y$ is a factor map between distal systems, we have that
$\pi^*(E_d^{\text{top}}(X))=E_d^{\text{top}}(Y)$ for every $d\in \N$.

\subsection{Enveloping
semigroups of systems of order $d$}
The results in this subsection
have been proven in \cite{SD14} for $G=\Z$.
When $G$ is abelian, it is easy to generalize the proof.
So we refer \cite{SD14} for the proofs.

\begin{lemma}\label{ellisgroup-of-order-d}
Let $(X,G)$ be a system of order $d$.
  Then, its enveloping semigroup is $d$-step top-nilpotent.
\end{lemma}

\begin{lemma}\label{uniform}
  Let $(X,G)$ be a minimal distal system.
If $E(X)$ is $d$-step top-nilpotent,
then $E_{d}^{\text{top}}(X)$ is a compact group of automorphisms of $(X,G)$
in the uniform topology.
\end{lemma}

\subsection{Equivalence relations generated by enveloping semigroups}
In this subsection,
let $d\geq2$ be an integer and let $(X,G)$ be a minimal system
with a $d$-step top-nilpotent enveloping semigroup $E(X)$.
For $j=1,\ldots,d$,
let
\[
\mathbf{R}_j(X)=\{(x,px):x\in X,p\in E_{j+1}^{\mathrm{top}}(X)\}.
\]
\begin{lemma}\label{euqi}

$\mathbf{R}_j(X)$ is a closed invariant equivalence relation,
and the factor $X_{j}=X/\mathbf{R}_j(X)$ has a $j$-step top-nilpotent
enveloping semigroup.
  Moreover it is the maximal factor of $X$ with this property and consequently
  $X_j$ is an extension of $X/\mathbf{RP}^{[j]}(X)$.
\end{lemma}

\begin{proof}
The result is trivial for $j=d$ as $\mathbf{R}_d(X)=\Delta$.

Let $j\in \{1,\ldots,d-1\}$.
We first show that
$\mathbf{R}_j(X)$ is a closed invariant equivalence relation.
The relation $\mathbf{R}_j(X)$ is $G$-invariant
  as every element of $E(X)$ commutes with $G$.

  Let $p_1,p_2\in E_{j+1}^{\mathrm{top}}(X)$ and $x\in X$, then $(x,p_1 x),(x,p_2x)\in \mathbf{R}_j(X)$.
  By Lemma \ref{subgroup} $E_{j+1}^{\mathrm{top}}(X)$ is a group,
  thus $p_1p_2^{-1}\in E_{j+1}^{\mathrm{top}}(X)$
  and $(y,p_1p_2^{-1} y)\in \mathbf{R}_j(X)$ for any $y\in X$.
  Particularly, let $y=p_2x$, then $(p_1x,p_2x)\in \mathbf{R}_j(X)$.
  This shows that $\mathbf{R}_j(X)$ is an equivalence relation.

We next show that $\mathbf{R}_j(X)$ is closed.

  Let $\{(x_n,p_n x_n)\}_{n\in \N}\subset \mathbf{R}_j(X)$
with $(x_n,p_n x_n) \to (x,y)$ as $n\to \infty$ for some $y\in X$.
  It suffices to show $(x,y)\in \mathbf{R}_j(X)$.
  Without loss of generality, by taking subsequence
we may assume that $p_n\to p$ as $n\to \infty$ for some $p\in E_{j+1}^{\mathrm{top}}(X) $.
  As $(X,G)$ is minimal, for every $n\in \N$
  there is some $q_n\in E(X)$ with $x_n=q_n x$.
  Without loss of generality, assume that $q_n \to q$,
  then $q_n x\to qx$ and $qx=x$.
  By property (3) of Proposition \ref{cube-group-closed},
  there is some element $r\in E_{j+2}^{\mathrm{top}}(X)$ such that $p_n q_n \to pqr.$
  Let $u=pqrq^{-1}=p[q,r]r$, then $u\in  E_{j+1}^{\mathrm{top}}(X)$ and
  \[
  (x_n,p_n x_n)=(q_nx,p_nq_nx)\to (qx,pqrx)=(qx,uqx)=(x,ux),
  \]
  as $n\to \infty$.
  This implies that $y=ux$ and $(x,y)\in \mathbf{R}_j(X)$ as was to be shown.

Now we can build factors by letting $X_j=X/\mathbf{R}_j(X),j=1,\ldots,d$.
  Let $\pi_j: X\to X_{j}$ be the factor map.
Let $v\in E_{j+1}^{\text{top}}(X)$, then we have $\pi_j(x)=\pi_j(vx)=\pi_j^*(v)\pi_j(x)$ for every $x\in X,$
  and thus $\pi_j^*(v)=\mathrm{id}_{X_j}$.
  This shows that $E_{j+1}^{\text{top}}(X_{j})$ is trivial.

  Let $(Z,G)$ be a factor of $(X,G)$ with a
  $j$-step top-nilpotent enveloping semigroup
  and let $\phi:X\to Z$ be the factor map.
 As $\phi^*(E_{j+1}^{\text{top}}(X))=\mathrm{id}_Z$,
 then for $u\in E_{j+1}^{\text{top}}(X)$, we have $\phi(ux)=\phi^*(u)\phi(x)=\phi(x)$
  and therefore $\phi$ can be factorized through $X_{j}$.

By Lemma \ref{ellisgroup-of-order-d},
 $X/\mathbf{RP}^{[j]}(X)$ has a $j$-step top-nilpotent
  enveloping semigroup, we deduce that
  $X_{j}$ is an extension of $X/\mathbf{RP}^{[j]}(X)$.
 \end{proof}

\begin{lemma}\label{iso--}
For $j=2,\ldots,d$, the extension $X_{j}\to X_{j-1}$ is isometric.
\end{lemma}

\begin{proof}
 Let $j\in \{2,\ldots,d\}$.
 By Lemma \ref{euqi}, $X_j$ has
 a $j$-step top-nilpotent enveloping semigroup.
 Let $X_{j-1}'=X_j/\mathbf{R}_{j-1}(X_j)$,
 then $X_{j-1}'$ has
 a $(j-1)$-step top-nilpotent enveloping semigroup and
 it is also a factor of $X_{j-1}$.

 By Lemma \ref{uniform},
 the extension $X_j\to X_{j-1}'$ is a group extension and thus it is an isometric extension.
 From this, we deduce that the extension $X_j\to X_{j-1}$ is also an isometric extension.
\end{proof}

\subsection{Proof of Theorem \ref{main-thm1}}
In this subsection, we will show Theorem \ref{main-thm1}.
Before giving the proof, we need some preparations.
\begin{lemma}\label{equal-ellis}
  Let $(X,G)$ be a dynamical system with a $d$-step top-nilpotent enveloping semigroup $E(X)$.
  For integer $l\geq 2^d$, let $E(X^{[l]},\mathcal{G}^{[l]})$ be
  the enveloping semigroup of the system $(X^{[l]},\mathcal{G}^{[l]})$.
 Then we have
 \[
 E_j^{\mathrm{top}}(X^{[l]},\mathcal{G}^{[l]}) \subset \mathcal{C}_j^{[l]}(E(X)), \; j=1,\ldots,d,
 \]
 where $\mathcal{C}_j^{[l]}(E(X)) $ is the group spanned by
\[
\{g^{(F)}:g\in E_i^{\mathrm{top}}(X), F\subset \{0,1\}^l,  \mathrm{codim}(F)=2^i,\;  i=j,\ldots, d \}.
\]
  In particular, $E(X^{[l]},\mathcal{G}^{[l]})$ is $d$-step top-nilpotent.
\end{lemma}

\begin{proof}
By Theorem \ref{distal-ellis-group} the system $(X,G)$ is distal, so is
the product system
\[
(X,G)^{[l]}=\underbrace{(X,G)\times \cdots \times (X,G)}_{2^l \;\mathrm{times}}.
\]
Hence $E((X,G)^{[l]})=E(X)^{[l]}$ is
 a closed group in $(X^{[l]})^{X^{[l]}}$.
It follows from property (2) of Proposition \ref{cube-group-closed} that
 $\mathcal{C}_j^{[l]}(E(X))$ is closed in $E(X)^{[l]}$ for every $j$ and
 thus it is also closed in $(X^{[l]})^{X^{[l]}}$.

 Recall that $E(X^{[l]},\mathcal{G}^{[l]})$ is the closure of $\mathcal{G}^{[l]}$ in $(X^{[l]})^{X^{[l]}}$
and $\mathcal{G}^{[l]}\subset \mathcal{C}_1^{[l]}(E(X))$
by the definition of the Host-Kra cubegroups,
we have
\[
E(X^{[l]},\mathcal{G}^{[l]}) \subset \mathcal{C}_1^{[l]}(E(X)).
\]
By induction and property (1) of Proposition \ref{cube-group-closed}, we obtain that
  \[
 E_j^{\mathrm{top}}(X^{[l]},\mathcal{G}^{[l]}) \subset \mathcal{C}_j^{[l]}(E(X)), \; j=1,\ldots,d,
 \]
 hence
  $ E(X^{[l]},\mathcal{G}^{[l]})$ is $d$-step top-nilpotent.
  \end{proof}

The following lemma is a simple observation.

\begin{lemma}\label{reduction}
  Let $(X,G)$ be a distal system and $d\in \N$.
  Let $Y$ be a subsystem of $X$, then for any $p\in E_d^{\mathrm{top}}(Y)$,
  there exists $p'\in E_d^{\mathrm{top}}(X)$ such that $p'|_Y=p$.
  That is, $p'(y)=p(y)$ for every $y\in Y$.
\end{lemma}

As a consequence of Lemmas \ref{equal-ellis} and \ref{reduction}, we have the following corollary.
\begin{cor}\label{subsystem-nilpotent}
  Let $d\in \N$ and let $(X,G)$ be a
  topological system with a $d$-step top-nilpotent enveloping semigroup.
  If $Y$ is a subsystem of $X$,
  then the enveloping semigroup of $Y$ is also $d$-step top-nilpotent.
\end{cor}

Now we are able to show Theorem \ref{main-thm1}.

\begin{proof}[Proof of Theorem \ref{main-thm1}]
We show it by induction on the nilpotency class $d$.

When $d=1$, it follows from Theorem \ref{distal-ellis-group} immediately.

Let $d>1$ be an integer and
suppose the statement is true for every $j=1,\ldots,d-1$.
Now let $(X,G)$ be a minimal system with a $d$-step
top-nilpotent enveloping semigroup and let $E(X)$ be its enveloping semigroup.
By Lemma \ref{euqi}, the factor $X_{d-1}=X/\mathbf{R}_{d-1}(X)$ has
a $(d-1)$-step top-nilpotent enveloping semigroup.
By inductive hypothesis, $X_{d-1}$ is a system of order $\infty$.
Thus by Remark \ref{infi}, we obtain that
\begin{equation}\label{include}
\mathbf{RP}^{[\infty]}(X)  \subset \mathbf{R}_{d-1}(X)=\{(x,px):x\in X, p\in E_d^{\mathrm{top}}(X)  \}.
\end{equation}

Suppose for a contradiction that $\mathbf{RP}^{[\infty]}(X)$ is nontrivial.
Choose $(x,y)\in \mathbf{RP}^{[\infty]}(X)$ with $x\neq y$.
By (\ref{include}), there is some element $p\in E_d^{\mathrm{top}}(X)$ with $y=px$.

For integer $l>2^d$,
let $\widetilde{E}^{[l]}$ be the enveloping semigroup of the system $(\mathbf{Q}^{[l]}(X),\mathcal{G}^{[l]})$.
It follows from Lemma \ref{equal-ellis} and Corollary \ref{subsystem-nilpotent} that $\widetilde{E}^{[l]}$ is also $d$-step
top-nilpotent.
As the system $(\mathbf{Q}^{[l]}(X),\mathcal{G}^{[l]})$ is minimal,
by Lemma \ref{euqi}
we obtain that the factor
\[
\mathbf{Q}^{[l]}(X)/\mathbf{R}_{d-1}(\mathbf{Q}^{[l]}(X))
\]
has a $(d-1)$-step top-nilpotent enveloping semigroup.
Again by inductive hypothesis, it is a system of order $\infty$.
It follows from Theorem \ref{infinity-step}
 that the maximal factor of order $\infty$ of $(\mathbf{Q}^{[l]}(X),\mathcal{G}^{[l]})$
is $(\mathbf{Q}^{[l]}(X_\infty),\mathcal{G}^{[l]})$, where $X_\infty= X/\mathbf{RP}^{[\infty]}(X)$.
So we have the following factor maps:
\[
\Q^{[l]}(X)\to \Q^{[l]}(X_\infty)\to \mathbf{Q}^{[l]}(X)/\mathbf{R}_{d-1}(\mathbf{Q}^{[l]}(X)).
\]

Let $\mathbf{x}=(y,x^{[l]}_*)=(px,x^{[l]}_*)$, then $\mathbf{x}\in \Q^{[l]}(X)$.
Moreover the images of the points $x^{[l]}$ and $\mathbf{x}$
in $\Q^{[l]}(X_\infty)$ are equal.
Recall that
\[
\mathbf{R}_{d-1}(\mathbf{Q}^{[l]}(X))=\{(\mathbf{y},\mathbf{q}\mathbf{y}):
\mathbf{y}\in\mathbf{Q}^{[l]}(X),\mathbf{q}\in (\widetilde{E}^{[l]})_d^{\mathrm{top}}\}.
\]
Hence there exists some element $\mathbf{q}\in (\widetilde{E}^{[l]})_d^{\mathrm{top}}$
such that
\begin{equation}\label{eee}
\mathbf{x}=\mathbf{q} x^{[l]}.
\end{equation}
As $(\Q^{[l]}(X),\mathcal{G}^{{l}})$ is a subsystem of $(X^{[l]},\mathcal{G}^{[l]})$,
by Lemma \ref{reduction} there is some element
 $\mathbf{p}\in  E_d^{\mathrm{top}}(X^{[l]},\mathcal{G}^{[l]})$
such that $\mathbf{q}\mathbf{y}=\mathbf{p}\mathbf{y}$ for all $\mathbf{y}\in \Q^{[l]}(X)$.
In particular, $\mathbf{q}x^{[l]}=\mathbf{p}x^{[l]}$.
So by (\ref{eee}), we have that
\begin{equation}\label{equal}
(px,x^{[l]}_*)=\mathbf{x}=
\mathbf{q}x^{[l]}=\mathbf{p} x^{[l]}.
\end{equation}
It follows from Lemma \ref{equal-ellis} that
  \[
 E_d^{\mathrm{top}}(X^{[l]},\mathcal{G}^{[l]}) \subset \mathcal{C}_d^{[l]}(E(X)).
 \]
Let $\mathbf{p}=(p_\epsilon:\epsilon\in \{0,1\}^l)$, then $\mathbf{p}\in \mathcal{C}_d^{[l]}(E(X))$.
By property (4) of Proposition \ref{cube-group-closed}
\begin{equation}\label{prop3}
\prod_{\epsilon\in \{0,1\}^l}p_\epsilon^{(-1)^{|\epsilon|}}=\mathrm{id}.
\end{equation}

Let us return to (\ref{equal}).
For every $\epsilon\in \{0,1\}^l$,
by considering the projection on the $\epsilon$-component,
we get that $p_{\vec{0}}x=px$ and $p_\epsilon x=x$ otherwise.
Notice that elements $p,p_\epsilon,\epsilon\in\{0,1\}^l$ all belong to $ E_d^{\mathrm{top}}(X)$
which is included in the center of $E(X)$,
we deduce that
$p_{\vec{0}}=p$ and $p_\epsilon=\mathrm{id}$ otherwise.
It is a contradiction as on this case the element $\mathbf{p}$ does not satisfy (\ref{prop3}).
We conclude that $X$ is a system of order $\infty$.

This completes the proof.
\end{proof}

\section{Furstenberg tower of minimal nilsystems}
Let $\pi:X\to Y$ be a factor map between minimal systems.
Then the family of factors of $X$ above $Y$ that are isometric extensions of $Y$
admits a maximal element, called the \emph{maximal isometric extension} of $Y$ below $X$.
In this section,
we show that for integer $s\geq 3$ and a minimal $s$-step nilsystem $(X,T)$,
the maximal isometric extension of $X/\mathbf{RP}^{[d]}(X)$ below $X$
is $X/\mathbf{RP}^{[d+1]}(X),d=1,\ldots,s-2$.


\subsection{Maximal isometric extensions}
We start with the following characterizations of maximal isometric extension.

\begin{definition}\label{def: maximal isometric extension}
  Let $(X,T)$ be a minimal system and let $R\subset X\times X$ be a closed invariant equivalence relation,
  we define
  \[
  Q(R)=\bigcap_{k=1}^\infty  \overline{\bigcup_{n=1}^\infty (T\times T)^{-n}\Delta_{\frac{1}{k}}\cap R},
  \]
  where $\Delta_{\frac{1}{k}}=\{(x,y)\in X\times X: \rho(x,y)<\frac{1}{k}\}$.
\end{definition}

\begin{lemma}\label{intersection}
  Let
$\xymatrix{X_1 \ar@/^1pc/[rr]^{\pi}  \ar[r]_{\phi}& X_2  \ar[r]&  X_3}$
 be factor maps between minimal systems $(X_i,T),i=1,2,3$.
 If $\pi$ is isometric, so is $\phi$.
\end{lemma}

\begin{proof}
Assume that $\pi$ is isometric.
Fix $\epsilon>0$.
Then there is some $\delta>0$
such that $\rho(T^nx_1,T^nx_2)<\varepsilon$ for all $n\in \mathbb{Z}$
whenever $\pi(x_1)=\pi(x_2)$ and $\rho(x_1,x_2)<\delta$.
Now let $(x,y)\in X_1$ with $\phi(x)=\phi(y)$ such that $\rho(x,y)<\delta$,
then $\pi(x)=\pi(y)$.
It follows that $\rho(T^nx,T^ny)<\varepsilon$ for all $n\in \mathbb{Z}$.
This shows that $\phi$ is isometric.
\end{proof}

\begin{lemma}\label{maximal-isometric}\cite[Chapter 7]{JA}
  Let $(X,T)$ be a minimal distal system and $R\subset X\times X$ be a closed invariant equivalence relation.
  Then the maximal isometric extension of $X/R$ below $X$ is $X/Q(R)$.
\end{lemma}

\begin{lemma}\label{factor--map}\cite[Appendix E.15]{JDV}
  Let $(X,T),(Y_1,T),(Y_2,T)$ be dynamical systems
  and let $\pi_i:(X,T)\to (Y_i,T),i=1,2$ be factor maps.
  If $R_{\pi_1}\subset R_{\pi_2}$, then there exists a factor map $\theta$ from $Y_1$ to $Y_2$,
  such that $\theta \circ \pi_1=\pi_2$.
 \end{lemma}

\begin{theorem}\label{isom}
  Let $(X,T)$ be a minimal distal system and $d\in \N$.
  Then the extension $X/\mathbf{RP}^{[d+1]}(X)\to X/\mathbf{RP}^{[d]}(X)$
  is isometric.
\end{theorem}

\begin{proof}
It follows from Lemma \ref{factor--map} and the fact $\mathbf{RP}^{[d+1]}(X)\subset \mathbf{RP}^{[d]}(X)$
that the extension $X/\mathbf{RP}^{[d+1]}(X)\to X/\mathbf{RP}^{[d]}(X)$ exists.
We may assume that this extension is nontrivial, otherwise there is nothing to prove.
That is, $\mathbf{RP}^{[d+1]}(X)$ is not equal to $\mathbf{RP}^{[d]}(X)$.
By Lemma \ref{maximal-isometric}, it suffices to show that
  $Q(\mathbf{RP}^{[d]}(X))\subset \mathbf{RP}^{[d+1]}(X)$.

Assume that $Q(\mathbf{RP}^{[d]}(X))$ is nontrivial,
and let $(x,y)\in Q(\mathbf{RP}^{[d]}(X))$ with $x\neq y$.
  Fix $\varepsilon>0$.
  By Definition \ref{def: maximal isometric extension}, there exist $(x',y')\in \mathbf{RP}^{[d]}(X)$ and $n\in \N$
  such that
  \[
  \rho(x,x')<\frac{\varepsilon}{4},\;  \rho(y,y')<\frac{\varepsilon}{4}\;\;
\text{and}\; \rho(T^nx',T^ny')<\varepsilon.
\]
There is $\delta>0$ such that $\rho(T^nu,T^nv)<\varepsilon$
whenever $u,v\in X$ with $\rho(u,v)<\delta<\varepsilon$.
As $(x',y')\in \mathbf{RP}^{[d]}(X)$,
there exist $x'',y''\in X$ and $\vec{n}=(n_1,\ldots,n_d)\in \Z^d$
such that
\[
\rho(x',x'')<\frac{\varepsilon}{4},\rho(y',y'')<\frac{\varepsilon}{4}
\; \text{and}\; \rho(T^{\vec{n}\cdot \epsilon}x'',T^{\vec{n}\cdot \epsilon}y'')<\delta,
\]
for all $\epsilon\in \{0,1\}^d\backslash \{\vec{0}\}$.

Put $\vec{m}=(\vec{n},n)$, then we have that
$\rho(x,x'')<\varepsilon,\rho(y,y'')<\varepsilon$ and
\[
\rho(T^{\vec{m}\cdot \epsilon}x'',T^{\vec{m}\cdot \epsilon}y'')<\varepsilon
\]
for all $\epsilon\in \{0,1\}^{d+1}\backslash \{\vec{0}\}$
which implies that $(x,y)\in \mathbf{RP}^{[d+1]}(X)$.

We conclude that
the extension $X/\mathbf{RP}^{[d+1]}(X)\to X/\mathbf{RP}^{[d]}(X)$
  is isometric.
\end{proof}

\subsection{Nilsystems}
We start by recalling some basic results in nilsystems.
For more details and proofs, see \cite{LA,PW}.
If $G$ is a nilpotent Lie group, let $G^0$ denote the connected component of its
unit element $1_G$.
Then $G^0$ is an open, normal subgroup of $G$.
In the sequel,
$s\geq3$ is an integer and $(X=G/\Gamma,T)$ is a minimal $s$-step nilsystem.
We let $\tau$ denote the element of $G$ defining the transformation $T$.

If $(X,T)$ is minimal, let $G'$ be the subgroup of $G$
spanned by $G^0$ and $\tau$ and let $\Gamma'=\Gamma\cap G'$,
then we have that $G=G'\Gamma$.
Thus the system $(X,T)$ is conjugate to the system $(X',T')$
where $X'=G'/\Gamma'$ and $T'$ is the translation by $\tau $ on $X'$.
Therefore,
without loss of generality, we can restrict to the case that $G$ is spanned by $G^0$
  and $\tau$.
 We can also assume that $G^0$ is simply connected
 (see for example \cite{LA} or \cite{AM} for the case that $G=G^0$
 and \cite{AL05} for the general case).
 This in turns implies that
 $G^0$ is \emph{divisible}, i.e., for any $g\in G^0, d\in \N$, there exists $h\in G^0$
such that $h^d=g$ (see for example \cite[Chapter 10, Corollary 9]{HK06}).
Recall that an element $g\in G$ is \emph{rational} (with respect to $\Gamma$)
if $g^n\in \Gamma$ for some $n\in \mathbb{N}$.
Let $\pi :G\to X$ be the natural projection.

For $r=1,\ldots,s-1$, define $Z_r=G/(G_{r+1}\Gamma)$.
Let $\pi_r:X \to Z_r$ be the quotient map.
The factors $Z_r$ define a decreasing sequence of factors between nilsystems,
starting with $X=Z_s$ and ending with $Z_1$:
\begin{equation}\label{factormap}
  Z_s\to Z_{s-1}\to \ldots\to Z_r\to Z_{r-1}\to \ldots \to Z_2\to Z_1.
\end{equation}

Refer \cite{PD13} or \cite{HK06} for the following description of maximal factors of higher order
 of a minimal nilsystem.

\begin{lemma}\label{nilfact}
  Let $(X,T)$ be a minimal $s$-step nilsystem.
 For $r=1,\ldots,s$, if $X_r$ is the maximal factor of order $r$ of $X$,
  then $X_r$ has the form $G/(G_{r+1}\Gamma)$, endowed with the translation by the projection
of $\tau $ on $G/G_{r+1}$.
\end{lemma}

  Lemma \ref{nilfact} tells us that for every $r=1,\ldots,s-1$,
$Z_r$ defined in (\ref{factormap}) is indeed the maximal
factor of order $r$ of $X$.
By Theorem \ref{isom}, we obtain that the extension $Z_{r+1}\to Z_r$ is isometric.
It suffices to show that such extension is maximal.

First, we choose a metric $d_G$ on the group $G$ that defines its topology.
For the moment, we only assume that this distance is invariant under right translations,
meaning that for all $g,g',h\in G$,
\[
d_G(gh,g'h)=d_G(g,g').
\]
The nilmanifold $X$ is endowed with the quotient distance,
meaning that $x,y\in X$,
\[
d_X(x,y)=\inf\{d_G(g,h):\pi(g)=x \; \text{and }\; \pi(h)=y\}.
\]
In other words, for all $g,h\in G$,
we have
\begin{equation}\label{00}
d_X(\pi(g),\pi(h))=\inf_{\alpha,\beta\in \Gamma}d_G(g\alpha,h\beta)=\inf_{\gamma\in \Gamma}d_G(g,h\gamma).
\end{equation}

\begin{lemma}\label{good-property}
  For any $g,h\in G_{s-1},r\in G$ and $n\in \N$,
  we have
  \[
  [gh,r]=[g,r][h,r] \quad \text{and} \quad [g^n,r]=[g,r]^n=[g,r^n].
  \]
  The maps
\[
G\to G,\quad \cdot \mapsto [g,\cdot],\quad \cdot \mapsto [\cdot,g]
\]
and
\[
 G_{s-1}\to G_{s-1},\quad \cdot \mapsto [\cdot,r]\quad \cdot \mapsto [r,\cdot]
\]
  are all continuous.
\end{lemma}
\begin{proof}
  As $[h,r]$ belongs to $G_s$
and thus belongs to the center of $G$,
we have
  \[
  [gh,r]=ghrh^{-1}g^{-1}r^{-1}=g(hrh^{-1}r^{-1})rg^{-1}r^{-1}=[g,r][h,r].
  \]
  Moreover, for any integer $n\geq1$,
  $[g^n,r]=[g,r]^n=[g,r^n]$.

  Now let $u,u'\in G_{s-1}$, then
  \begin{align*}
    d_G([u,g],[u',g])& =d_G([u,g],[uu^{-1}u',g]) \\
     & =d_G([u,g],[u^{-1}u',g][u,g]) \\
     &  =d_G(1_G,[u^{-1}u',g]) \quad  (\text{by right-invariance})\\
     &\leq d_G(u,u')+d_G(gug^{-1},gu'g^{-1}).
  \end{align*}
  As $G$ is a Lie group,
the map $G\to G,x\mapsto gxg^{-1}$ is continuous.
It follows that the map $G_{s-1}\to G_{s-1},\cdot \mapsto [\cdot,g]$
  is also continuous.

 By similar argument,
 we obtain the result.
\end{proof}


\begin{lemma}\cite[Corollary 1.14]{AL06}\label{rational-dense}
  A closed subgroup $H$ of $G$ is rational if and only if rational elements
of $G$ are dense in $H$.
\end{lemma}

The following result can be found in \cite{HKM02},
here we give a direct proof.
 \begin{lemma}\label{key}
For every $\varepsilon>0$ and $t\in G_s$, there exist
  $h\in G_{s-1},\theta\in G_s\cap \Gamma$
  and $k\in \N$ such that
  \begin{equation}\label{kkkk}
  d_G([h,\tau^k],t\theta)<\varepsilon.
  \end{equation}
 \end{lemma}

\begin{proof}
Fix  $\varepsilon>0$ and $t\in G_s$.
Let $u\in G_{s-1}$ and $v\in G$ with $[u,v]=t$.
As $G$ is spanned by $G^0$ and $\tau$,
there exist $v_1\in G^0$ and $n\in \N$ with
$v=v_1 \tau^n.$

Recall that $G_{s-1}$ is a rational subgroup of $G$,
thus by Lemma \ref{rational-dense},
there is some rational element $u_1\in G_{s-1}$ such that 
   \[
       d_G([u,v_1\tau^n],[u_1,v_1\tau^n])
  <\frac{\varepsilon}{2}.
   \]

   Let $l\in \mathbb{N}$ such that $u_1^l\in \Gamma$.
 Recall that $G^0$ is divisible, choose $v_2\in G^0$ with $v_2^l=v_1$.
By Lemma \ref{good-property},
we get
\[
  [u_1,v_1\tau^n]=[u_1,v_2^l\tau^n]=[u_1^l,v_2][u_1,\tau^n].
\]

By continuity of the map $\cdot \mapsto [u_1^l,\cdot]$
and by minimality of the system $(X,T)$,
there exist $m\in \mathbb{N}$
and $\gamma\in \Gamma$ such that
\[
  d_G([u_1^l,v_2],[u_1^l,\tau^m \gamma])<\frac{\varepsilon}{2}.
\]

Now we have
  \begin{align*}&d_G(t,[u_1^l,\tau^m\gamma][u_1,\tau^n])\\
 =&d_G([u,v_1\tau^n],[u_1^l,\tau^m\gamma][u_1,\tau^n])& \\
 \leq &d_G([u,v_1\tau^n],[u_1,v_1\tau^n])+d_G([u_1,v_1\tau^n],[u_1^l,\tau^m\gamma][u_1,\tau^n]) \\
      <& d_G([u_1,v_1\tau^n],[u_1^l,\tau^m\gamma][u_1,\tau^n])+\frac{\varepsilon}{2}
     & \\
       =&d_G([u_1^l,v_2][u_1,\tau^n],[u_1^l,\tau^m\gamma][u_1,\tau^n])+\frac{\varepsilon}{2}
     & \\
     =& d_G([u_1^l,v_2],[u_1^l,\tau^m \gamma])+\frac{\varepsilon}{2}<\varepsilon.&
  \end{align*}

Again by Lemma \ref{good-property}, we obtain that
\[
[u_1^l,\tau^m\gamma][u_1,\tau^n]=[u_1,\tau^{ml+n}][u_1^l,\gamma].
\]

Set $h=u_1,k=ml+n,\theta=[u_1^l,\gamma]^{-1}\in G_s\cap \Gamma$,
then
$d_G([h,\tau^k],t\theta)<\varepsilon$.
\end{proof}

\begin{remark}\label{small}
Indeed, in (\ref{kkkk}) of Lemma \ref{key},
we can choose $h$ small enough satisfying the inequality.
Now assume that $h,\theta,k$ have been chosen satisfying (\ref{kkkk}).
Let $K$ be the closed ball $\{g\in G^0:d_G(g,1_G)\leq 2d_G(h,1_G)\}$ in $G^0$ and let
the Lie algebra $\mathfrak{g}$ of $G^0$ endowed with a norm $||\cdot||$.
Let $L$ be a closed ball in $\mathfrak{g}$ centered at 0 such that $\{  \exp {\xi}:\xi\in L \}\supset K$.
Since the exponential map exp is a diffeomorphism from $\mathfrak{g}$ onto $G^0$,
we have that the restriction of the exponential map $L$ is Lipschitz.
Thus there exists a constant $C>0$ such that for every $\xi \in L$,
\[
C^{-1}||\xi||\leq d_G(\exp (\xi),1_G) \leq C ||\xi||.
\]
Writing $h=\exp(\xi)$ for some $\xi\in L$ and setting $h_m=\exp(\xi/m)\in G_{s-1}$,
it follows that $(h_m)^{m}=\exp(\xi)=h$ and
\begin{equation}\label{large}
  d_G(h_m,1_G)\leq C||\xi/m||=C||\xi||/m\leq C^2d_G(h,1_G)/m.
\end{equation}
In (\ref{large}), let integer $m$ large enough such that $d_G(h_m,1_G)<\varepsilon$,
then by Remark \ref{good-property}, \[
d_G([h_m,\tau^{mk}],t\theta)=
d_G([(h_m)^m,\tau^{k}],t\theta)=d_G([h,\tau^{k}],t\theta)<\varepsilon.
\]
\end{remark}

 We emphasize that there is no other restriction on $k$ in (\ref{kkkk})
 and when using Lemma \ref{key},
 we always take $h$ in (\ref{kkkk}) small enough.
\begin{lemma}\label{not-iso}
  The extension $\pi_{s-2}:X\to Z_{s-2}$ is not isometric.
\end{lemma}

\begin{proof}
Suppose for a contradiction that the extension $\pi_{s-2}$ is isometric.

Let $t\in G_s$ with $\inf_{\gamma\in \Gamma}d_G(t,\gamma)>0$.
We claim that $(\Gamma,t\Gamma)$ is a proximal pair.

Fix $\varepsilon>0$.
As $\pi_{s-2}$ is isometric,
there is some $\varepsilon>\delta>0$
such that whenever $g,h\in G$
with $\pi_{s-2}(g\Gamma)=\pi_{s-2}(h\Gamma)$
and $d_X(g\Gamma,h\Gamma)<\delta$,
 we have
 $d_X(T^n g\Gamma,T^n h\Gamma)<\varepsilon$ for all $n \in \Z$.
For such $\delta$,
it follows from Lemma \ref{key}
and Remark \ref{small}
that there exist $u\in G_{s-1},\theta\in G_s\cap \Gamma$ and $n_0\in \Z$ such that
$d_G(u,1_G)<\delta$ and $d_G([u^{-1},\tau^{n_0}],t\theta)<\delta.$

As $t,\theta$ belong to the center of the group $G$,
we have that
\begin{align*}
d_G(\tau^{n_0} u,\tau^{n_0} t\theta)&=d_G(u[u^{-1},\tau^{n_0}],t\theta)&\\
&\leq d_G(u[u^{-1},\tau^{n_0}],[u^{-1},\tau^{n_0}])+
 d_G([u^{-1},\tau^{n_0}],t\theta) \nonumber&\\
 &= d_G(u,1_G)+
 d_G([u^{-1},\tau^{n_0}],t\theta) \nonumber\\
 &< 2\delta.
\end{align*}
It is easy to see that
$\pi_{s-2}(\Gamma)=\pi_{s-2}(u\Gamma)$
and $d_X(\Gamma,u\Gamma)\leq d_G(1_G,u)<\delta$,
then since $\pi_{s-2}$ is isometric
\[
 d_X(T^n \Gamma,T^n u \Gamma)<\varepsilon
\]
for all $n\in \Z$.
By the definition of the metric $d_X$,
we can choose $\gamma_0\in \Gamma$ such that
\begin{equation*}
d_G(\tau^{n_0},\tau^{n_0} u\gamma_0)=d_X(T^{n_0}\Gamma,T^{n_0} u\Gamma)<\varepsilon.
\end{equation*}
Thus we have that
\begin{align*}
  d_X(T^{n_0}\Gamma,T^{n_0}t\Gamma)&\leq d_G(\tau^{n_0},\tau^{n_0}t\theta\gamma_0) & \\
   &\leq  d_G(\tau^{n_0},\tau^{n_0}u\gamma_0)  + d_G(\tau^{n_0}u\gamma_0,\tau^{n_0}t\theta\gamma_0) & \\
   &=d_G(\tau^{n_0},\tau^{n_0}u\gamma_0)  + d_G(\tau^{n_0}u,\tau^{n_0}t\theta)&\\
   &<\varepsilon+2\delta<3\varepsilon,
\end{align*}
which implies that $(\Gamma,t\Gamma)$ is a proximal pair.

This is a contradiction as the system $(X,T)$ is distal.
This shows that the extension $\pi_{s-2}$ is not isometric.
\end{proof}

\begin{lemma}\label{tower}
   The maximal isometric extension of $Z_{s-2}$ below $X$ is $Z_{s-1}.$
\end{lemma}

\begin{proof}
Recall that $\pi_r:X \to Z_r$ is the quotient map $r=1,\ldots,s-1$.
By Lemma \ref{not-iso}, $Q(R_{\pi_{s-2}})$ is nontrivial.
By Lemma \ref{maximal-isometric} and Theorem \ref{isom} we have the following factor maps:
\[
X\to X/Q(R_{\pi_{s-2}})\to Z_{s-1}=X/R_{\pi_{s-1}}\to Z_{s-2}=X/R_{\pi_{s-2}}.
\]
It suffices to show that
$R_{\pi_{s-1}}\subset Q(R_{\pi_{s-2}})$.
We need the following claims.

\medskip

\noindent {\bf Claim 1}:
Let $u,v\in G$ with
$(u\Gamma,v\Gamma)\in Q(R_{\pi_{s-2}})\backslash \Delta$,
then$(u\Gamma,tv\Gamma)\in Q(R_{\pi_{s-2}})$ for all $t\in G_s$.

\begin{proof}[Proof of Claim 1]
Let $u,v\in G$ with
$(u\Gamma,v\Gamma)\in Q(R_{\pi_{s-2}})\backslash \Delta$.
Let ${t\in G_s}$ and assume that $t$ does not belong to $\Gamma$, otherwise there is nothing to prove.

Fix $\varepsilon>0$.
Since $(u\Gamma,v\Gamma)\in Q(R_{\pi_{s-2}})$,
there exist $(u'\Gamma,v'\Gamma)\in R_{\pi_{s-2}}$
and \textcolor{red}{infinitely many}\footnote{This means that we can choose integer $m$ large enough  to satisfy the condition
in Lemma \ref{key}.}
$m\in \mathbb{N}$
such that
$ d_X(u\Gamma,u'\Gamma)<\varepsilon,d_X(v\Gamma,v'\Gamma)<\varepsilon$
and
\begin{equation}\label{infi-oo}
  d_X(T^mu'\Gamma,T^mv'\Gamma)<\varepsilon.
\end{equation}
As $t\in G_s$,
by Lemma \ref{key}
and Remark \ref{small}
there exist ${h\in G_{s-1},\theta\in G_s\cap \Gamma}$
and $n\in \N$ which also satisfies (\ref{infi-oo}) such that
  \begin{equation}\label{6}
    d_G(h,1_G)<\varepsilon,\; d_G([h^{-1},u'],1_G)<\varepsilon,
  \;\text{and}\; d_G([h^{-1},\tau^n],t\theta)<\varepsilon.
  \end{equation}
As $(u'\Gamma,v'\Gamma)\in R_{\pi_{s-2}} $,
  we have $u'^{-1}v'\in G_{s-1}\Gamma$ and
   \[
   (u'h)^{-1}tv'=th^{-1}u'^{-1}v'\in G_{s-1}\Gamma,
   \]
which implies $(u'h\Gamma,tv'\Gamma)\in R_{\pi_{s-2}}$.
  We next show that
   \begin{enumerate}
    \item $d_X(u\Gamma,u'h\Gamma)<3\varepsilon$;
     \item $d_X(tv\Gamma,tv'\Gamma)<\varepsilon$;
      \item $d_X(T^nu'h\Gamma,T^ntv'\Gamma)<4\varepsilon$.
       \end{enumerate}
       From this, we get that $(u\Gamma,tv\Gamma)\in Q(R_{\pi_{s-2}})$.

 \medskip

To estimate (1),
choose $\gamma,\gamma_1\in \Gamma$ such that
\begin{align}\label{99}
   & d_G(u,u'\gamma_1)= d_X(u\Gamma,u'\Gamma)<\varepsilon,
   &  \\
   &d_G(\tau^nu',\tau^nv'\gamma)=d_X(T^nu'\Gamma,T^nv'\Gamma)<\varepsilon.
\end{align}
By right-invariance of the metric $d_G$,
we have
\begin{align*}
d_X(u\Gamma,u'h\Gamma)&\leq  d_G(u,u'h\gamma_1)& \\
& = d_G(u,h[h^{-1},u']u'\gamma_1)& \\
  & \leq d_G(u,u'\gamma_1)+d_G(u'\gamma_1,h[h^{-1},u']u'\gamma_1) &\\
  &=d_G(u,u'\gamma_1)+d_G(1_G,h[h^{-1},u'])&\\
  &\leq d_G(u,u'\gamma_1)+d_G(1_G,[h^{-1},u'])+d_G(1_G,h)<3
  \varepsilon.  &\text{by} \;(\ref{6})\; \mathrm{and}\; (\ref{99})
\end{align*}

\medskip
To estimate (2),
   as $t$ belongs to the center of the group $G$,
   we have
 \[
 d_X(tv\Gamma,tv'\Gamma)=d_X(v\Gamma,v'\Gamma)<\varepsilon.
 \]

  \medskip

To estimate (3),
 we have
  \begin{align*}
     &d_X(T^nu'h\Gamma,T^ntv'\Gamma)&\\
     \leq& d_X(T^n u'h\Gamma,T^n u't\Gamma)+d_X(T^n u't\Gamma,T^n tv'\Gamma)&\\
      < &d_X(T^n u'h\Gamma,T^n u't\Gamma)+\varepsilon  & \text{by} \;(\ref{infi-oo})\\
     \leq& d_G(\tau^n u' h,\tau^n u' t\theta)+\varepsilon&\\
      =&d_G(\tau^n h[h^{-1},u']u',\tau^n t\theta u')+\varepsilon &  (t,\theta\in G_s)\\
 =&d_G( [h^{-1},u']\tau^n h ,\tau^n t \theta)+\varepsilon & ([h^{-1},u']\in G_s)&\\
 \leq& d_G([h^{-1},u']\tau^n h, \tau^nh)+d_G(\tau^n h,\tau^n t \theta)+\varepsilon&\\
 \leq &d_G([h^{-1},u'], 1_G)+d_G(h,1_G)+d_G([h^{-1},\tau^n],t\theta)+\varepsilon.&\\
 <&4\varepsilon. & \text{by } (\ref{6})
  \end{align*}

This shows Claim 1.
  \end{proof}

Let $\theta : X\to X/Q(R_{\pi_{s-2}})$ be the factor map
and let $\varphi=\theta\circ \pi:G\to X/Q(R_{\pi_{s-2}})$, then $\varphi$ is continuous.
Let $H=\{ g\in G: \varphi(g)=\varphi(1_G)\}.$
Clearly, $\Gamma\subset H\subset G_s \Gamma$.

\medskip

\noindent {\bf Claim 2}:
$(g\Gamma,gh\Gamma)\in Q(R_{\pi_{s-2}})$ for any $g\in G$ and $h\in H$.

\begin{proof}[Proof of Claim 2]
Let $g\in G$ and $h\in H$, then $(\Gamma,h\Gamma)\in Q(R_{\pi_{s-2}})$.
Since $H\subset G_s\Gamma$ and $G_s$ is in the center of $G$, we may assume that $h\in G_s$.

By minimality of the system $(X,T)$, there is some sequence $\{n_k\}_{k\in \N}\subset \Z$ with
\[
d_X(T^{n_k}\Gamma, g\Gamma)\to 0,k\to \infty.
\]
As $h$ belongs to the center of the group $G$,
we have
\[
d_X(T^{n_k}h\Gamma, gh\Gamma)=d_X(T^{n_k}\Gamma, g\Gamma)\to 0,k\to \infty,
\]
which implies $(g\Gamma,gh\Gamma)\in Q(R_{\pi_{s-2}})$ by
equivalence of $Q(R_{\pi_{s-2}})$.
\end{proof}

\noindent {\bf Claim 3}:
There is some element $h\in H\cap G_s$ such that $h\notin \Gamma$.

\begin{proof}[Proof of Claim 3]
Recall that $Q(R_{\pi_{s-2}})$ is nontrivial,
we can choose $u,v\in G$ with $(u\Gamma,v\Gamma)\in Q(R_{\pi_{s-2}})\backslash \Delta$.
Then $(u\Gamma,v\Gamma)\in R_{\pi_{s-1}}$ which implies that there is some $h\in G_s$ and $h\notin \Gamma$
such that $v\Gamma=uh\Gamma$.
It suffices to show that $h\in H$.

As the system $(X,T)$ is minimal, there is some sequence $\{m_k\}_{k\in \N}\subset \Z$ such that
\[
d_X(T^{m_k}u\Gamma ,\Gamma)\to 0,k\to \infty.
\]
As $h$ belongs to the center of the group $G$,
we have
\[
d_X(T^{m_k}uh\Gamma,h\Gamma)=d_X(T^{m_k}u\Gamma ,\Gamma)\to 0,k\to \infty,
\]
which implies $(\Gamma,h\Gamma)\in Q(R_{\pi_{s-2}})$ by
equivalence of $Q(R_{\pi_{s-2}})$.
This shows $h\in H$.
\end{proof}

We next show that $R_{\pi_{s-1}}\subset Q(R_{\pi_{s-2}})$.

Let $u,v\in G$ with $(u\Gamma,v\Gamma)\in R_{\pi_{s-1}}$,
then there exist $g\in G_s$ such that $v\Gamma=gu\Gamma$.
  By Claim 3,
  there is some element $h\in H\cap G_s$ such that $h\notin \Gamma$.
  By Claim 2, we get that $(u\Gamma,uh\Gamma)\in Q(R_{\pi_{s-2}})$ .
Notice that $u\Gamma$ and $uh\Gamma$
are distinct points by the choice of $h$,
so applying Claim 1 by taking $t=h^{-1}g\in G_s$,
we obtain that
\[
(u\Gamma,v\Gamma)=(u\Gamma,gu\Gamma)=
(u\Gamma,(h^{-1}g) uh \Gamma)\in  Q(R_{\pi_{s-2}}),
\]
  as was to be shown.

This completes the proof.
\end{proof}

Now we are able to show Theorem \ref{main-thm2}.

\begin{proof}[Proof of Theorem \ref{main-thm2}]

For $r=1,\ldots,s-1$, define $Z_r=G/(G_{r+1}\Gamma)$.
By Lemma \ref{nilfact}, we get that $Z_r$ is the maximal factor of order $r$ of $X$.
It suffices to show that the maximal isometric extension of $Z_{r}$ below $X$
is $Z_{r+1}$ for $r=1,\ldots,s-2$.

We first show that
the maximal isometric extension of $Z_{r}$ below $Z_{r+2}$
is $Z_{r+1}$ for $r=1,\ldots,s-2$.
As a matter of fact,
let $G'=G/G_{r+3},\phi:G\to G'$ be the quotient homomorphism,
$\Gamma'=\phi(\Gamma)$ and $\tau'=\phi(\tau)$.
Let $T'$ be the translation by $\tau'$ on $X'=G'/\Gamma'$.
Then for every $j\geq1$, we have that $\phi(G_j)=G_j'$.
Therefore $G'_{r+3}=\{1_{G'}\}$ and $(X',T')$ is an $(r+2)$-step nilsystem.
Moreover, since $(X',T')$ is a factor of $(X,T)$, it is minimal.
For $j=1,\ldots,r+2$, let $Z'_{j}=G'/(G'_{j+1}\Gamma')$.
Applying Lemma \ref{tower} for $s=r+2$,
we obtain that the maximal isometric extension of $Z'_{r}$ below $X'=Z_{r+2}'$ is $Z'_{r+1}$.
Recalling that $G'=G/G_{r+3}$,
we obtain that $Z_{i}$ is conjugate to $Z_{i}'$ for $i=1,\ldots,r+2$.
Thus the maximal isometric extension of $Z_{r}$ below $Z_{r+2}$ is $Z_{r+1}$.

\medskip

We now return to the proof of the theorem.
We show it by induction on $r$.

When $r=s-2$, it follows from Lemma \ref{tower}.

Let $r\leq s-3$ be an integer and
suppose the statement is true for all $j=r+1,\ldots, s-2$.
Assume that $Y$ is the maximal isometric extension of $Z_{r}$ below $X$,
then
there is a commutative diagram:
\[\xymatrix{
  X \ar[drr]_{} \ar[r]^{}
             &    Z_{r+2} \ar[r] &Z_{r+1} \ar[r]^{}  &Z_{r}  \\
             &   & Y \ar[u]^{}    \ar[ur]^{}         }
             \]
      By Lemma \ref{intersection}, the extension $Y\to Z_{r+1}$ is also isometric.
      By inductive hypothesis,
      the maximal isometric extension of $Z_{r+1}$ below $X$ is $Z_{r+2}$.
      This shows that $Y$ is a factor of $Z_{r+2}$ and thus $Y$ is an isometric extension of $Z_{r}$ below $Z_{r+2}$.
      Notice that the maximal isometric extension of $Z_{r}$ below $Z_{r+2}$ is $Z_{r+1}$,
     we deduce that $Y$ is a factor of $Z_{r+1}$ and so $Y=Z_{r+1}$  as was to be shown.

      This completes the proof.
\end{proof}

\section{Proofs of Theorems \ref{main-thm3} and \ref{main-thm4}}

In the final section, we give proofs of Theorems \ref{main-thm3} and \ref{main-thm4}.

\begin{lemma}\label{d-step-nil}
  Let $d\geq 2$ be an integer and let $(X,T)$ be a minimal nilsystem with a
$d$-step top-nilpotent enveloping semigroup,
then it is a $d$-step nilsystem.
\end{lemma}

\begin{proof}
 Let $d\geq 2$ be an integer and let $(X,T)$ be a minimal nilsystem with a
$d$-step top-nilpotent enveloping semigroup.
For $j=1,\ldots,d$, let $Z_j=X/\mathbf{RP}^{[j]}(X)$ and let $X_{j}=X/\mathbf{R}_{j}(X)$.
By Lemma \ref{euqi},
we obtain the following commutative diagram:
\begin{equation*}
\xymatrix
        {(X=X_d,T) \ar[d]    \ar[r]        &
     (X_{d-1},T) \ar[r]      \ar[d]       &
         \cdots \ar[r] &
        (X_2,T) \ar[d] \ar[r]  &
                (X_1,T)  \ar[d] &
        \\
 (Z_d,T) \ar[r]               &
     (Z_{d-1},T) \ar[r]         &
  \cdots      \ar[r]   &
     (Z_2,T)\ar[r] &
     (Z_1,T)
    }
\end{equation*}

We will show inductively that $Z_j=X_j,j=1,\ldots,d-1$.

Note that
$X_1$ is a factor of $X$,
 for which its
enveloping semigroup is abelian.
By Theorem \ref{distal-ellis-group},
$X_1$ is equicontinuous.
Since $Z_1$ is the maximal equicontinuous factor of $X$,
we obtain that $X_1$ is a factor of $Z_1$ and thus $X_1=Z_1$.
For integer $d\geq 3$,
let $j\in \{2,\ldots, d-1\}$.
Suppose the statement is true for all $k=1,\ldots,j-1$.
That is, $Z_k=X_k,k=1,\ldots,j-1$.
By Lemma \ref{iso--} the extension $X_j\to X_{j-1}$ is isometric.
On the other hand, by Theorem ~\ref{main-thm2}
 the maximal isometric extension of $Z_{j-1}$ below $X$
is $Z_j$.
This implies that $X_j$ is a factor of $Z_j$ and thus $Z_j=X_j$.

For any integer $d\geq2$,
we conclude that $Z_j=X_j,j=1,\ldots,d-1$.
We next show that $X=Z_d$ which implies that $X$ is a $d$-step nilsystem.

Suppose for a contradiction that $X$ is not a $d$-step nilsystem.
Then $\RP^{[d]}(X)$ is nontrivial and by Theorem \ref{main-thm2}
 the maximal isometric extension of $Z_{d-1}$ below $X$
is $Z_d$.
Again by Lemma \ref{iso--},
we obtain that the extension $X\to X_{d-1}=Z_{d-1}$ is isometric which implies that $X$ is a factor of $Z_d$.
It is a contradiction!

We conclude that $X$ is a $d$-step nilsystem.
\end{proof}

\begin{cor}\label{key-thm}
    Let $d\geq 2$ be an integer and let $(X,T)$ be a system of order $\infty$ with a
$d$-step top-nilpotent enveloping semigroup,
then it is a $d$-step pro-nilsystem.
\end{cor}

\begin{proof}
  By Theorem \ref{system-of-order}, there exists a sequence of minimal nilsystems
  $\{X_i\}_{i\in \N}$
  such that $X=\lim\limits_{\longleftarrow}\{ X_i\}_{i\in \N}$.

  As the enveloping semigroup of $(X,T)$ is $d$-step top-nilpotent,
we obtain that for every $i\in \N$,
the enveloping semigroup of $(X_i,T)$ is
also $d$-step top-nilpotent.
So by Lemma \ref{d-step-nil}, $(X_i,T)$ is a $d$-step nilsystem for every $i\in\N$.

 We conclude that $X$ is a $d$-step pro-nilsystem.
\end{proof}


Finally, we are able to show Theorems \ref{main-thm3} and \ref{main-thm4}.
\begin{proof}[Proof of Theorem \ref{main-thm3}]
When $d=1$, it follows from Theorem \ref{distal-ellis-group}.

When $d\geq2$,
it follows from Theorem \ref{main-thm1} and Corollary \ref{key-thm}.
\end{proof}


\begin{proof}[Proof of Theorem \ref{main-thm4}]
Let $d\geq 2$ be an integer.
When $k=d$, the result is trivial.

Let $k\in \{1,\ldots , d-1\}$.
It follows from Lemma \ref{euqi} that
$X/\mathbf{R}_k(X)$ is an extension of $X/\mathbf{RP}^{[k]}(X)$.

On the other hand, the enveloping semigroup of $X/\mathbf{R}_k(X)$ is $k$-step top-nilpotent,
thus by Theorem \ref{main-thm3} it is a $k$-step pro-nilsystem.
Recall that $X/\mathbf{RP}^{[k]}(X)$ is the maximal factor of order $k$ of $X$,
so $X/\mathbf{R}_k(X)$ is a factor of $X/\mathbf{RP}^{[k]}(X)$.

Following these facts, we deduce that $X/\mathbf{R}_k(X)=X/\mathbf{RP}^{[k]}(X)$
which implies
\[
\mathbf{RP}^{[k]}(X)=\{(x,px):x\in X,p\in   E_{k+1}^{\mathrm{top}}(X)\}.
\]

This completes the proof.
\end{proof}

\end{document}